\documentclass[a4paper,12pt]{article}

\usepackage[english]{babel}
\usepackage[english]{translator}
\usepackage[T1]{fontenc}
\usepackage[utf8]{inputenc}
\usepackage{amsmath}
\usepackage{amssymb}
\usepackage{amsfonts}
\usepackage{amstext}
\usepackage{amsthm}
\usepackage{bbm}
\usepackage{graphicx}
\usepackage{color}
\usepackage{psfrag}
\usepackage{subfigure}
\usepackage{float}
\usepackage{makecell}
\usepackage{enumerate}
\usepackage{enumitem}
\usepackage{todonotes}
\usepackage{hyperref}

\usepackage[
  hmarginratio={1:1}, 
  vmarginratio={1:1},   
  textwidth=460pt,       
  heightrounded,       
  bottom=3.65cm,
  top=3.65cm,
]{geometry}

\hypersetup{
  pdffitwindow=false,
  pdfhighlight=/O,
  pdfnewwindow,
  colorlinks=true,
  citecolor=red,       
  linkcolor=blue,        
  menucolor=blue,      
  urlcolor=blue, 
  pdfpagemode=UseOutlines,
  bookmarksnumbered=true,
  linktocpage,
  pdfkeywords={},
  pdfcreator={pdflatex},
  pdfproducer={LaTeX with hyperref}
}

\newcommand{\R}{\mathbb{R}}           
\newcommand{\Z}{\mathbb{Z}}       
\newcommand{\N}{\mathbb{N}}
\newcommand{\T}{\mathbb{T}^d}
\newcommand{\es}{\varepsilon}

\renewcommand{\i}{\mathrm{i}}

\renewcommand{\div}{\mathrm{div} \,}	
\renewcommand{\hom}{\mathrm{hom}}
\newcommand{\mubar}{\overline{\mu}_{\varepsilon}}

\newcommand{\Rdd}{\R^{d \times d}}
\newcommand{\Rspd}{\Rdd_{\mathrm{spd}}}
\newcommand{\Rsym}{\Rdd_{\mathrm{sym}}}
\newcommand{\A}{\mathcal{A}}
\newcommand{\Bx}{\Box^{-1}}
\newcommand{\bfmu}{\boldsymbol{\mu}}

\DeclareMathOperator*{\argmin}{arg\,min}

\theoremstyle{definition}
\newtheorem{definition}{Definition}[section]
\newtheorem{assumption}{Assumption}

\theoremstyle{plain}
\newtheorem{theorem}[definition]{Theorem}
\newtheorem{lemma}[definition]{Lemma}
\newtheorem{corollary}[definition]{Corollary}
\newtheorem{proposition}[definition]{Proposition}

\newenvironment{abstr}[1]{ \vspace{.05in}\footnotesize
	\parindent .2in
	{\upshape\bfseries #1. }\ignorespaces}{\par\vspace{.1in}}
\newenvironment{Abstract}{\begin{abstr}{Abstract}}{\end{abstr}}
\newenvironment{keywords}{\begin{abstr}{Key words}}{\end{abstr}}
\newenvironment{AMS}{\begin{abstr}{AMS subject classifications}}{\end{abstr}}

\allowdisplaybreaks

\begin{document}

\title{Optimization-Based Identification of Effective Coefficients for Wave Equations in Spatio-Temporal Metamaterials}

\author{Christian D\"oding$^{1,2}$, Vishnu Raveendran$^{1,3}$, Barbara Verf\"urth$^{1,4}$}
\setlength{\parindent}{0pt}
\date{}

\footnotetext[1]{Institute for Numerical Simulation, University of Bonn, 53115 Bonn, Germany}
\footnotetext[2]{e-mail: \textcolor{blue}{doeding@ins.uni-bonn.de}}
\footnotetext[3]{e-mail: \textcolor{blue}{raveendr@ins.uni-bonn.de}}
\footnotetext[4]{e-mail: \textcolor{blue}{verfuerth@ins.uni-bonn.de}}

\maketitle 

\begin{Abstract} 
    We study the identification of effective coefficients for wave equations in heterogeneous media. Such equations arise in the modeling of spatio-temporal metamaterials, where the underlying material properties exhibit variations in both space and time. While homogenization provides effective models in the asymptotic regime of vanishing microscopic scales, determining macroscopic parameters from observations of wave propagation remains challenging. We introduce an optimization-based method that identifies a constant effective coefficient by minimizing a cost functional. The approach is designed for situations in which the underlying space-time-dependent coefficient is unknown, while the solution is available in space and time by measurements. We extend optimization-based coefficient identification techniques from elliptic multiscale problems to wave equations and prove that, provided a homogenized limit exists, the identified coefficient converges to the homogenized coefficient as the microscopic scale tends to zero. Furthermore, we establish convergence of the corresponding effective solution towards the heterogeneous solution, including strong convergence in $L^2$ and weak convergence of first-order space and time derivatives. Numerical experiments further demonstrate the performance of the method for non-periodic space-time heterogeneous media, including cases for which a homogenized limit is not known to exist.
\end{Abstract}
	
\begin{keywords}
	Optimization-based identification; time-modulated wave equations; spatio-temporal metamaterials; heterogeneous media; multiscale problems; effective coefficients; homogenization.
\end{keywords}

\begin{AMS}
	 	35B27; 35L05; 35Q60; 74J25; 78-10; 78A40; 78M40
\end{AMS}

\section{Introduction}

Metamaterials are artificially engineered materials whose effective properties can exhibit phenomena not commonly observed in natural materials, including absorption, amplification, non-reciprocity, negative refraction, and cloaking; see \cite{Caloz20a,Caloz20,Cui10}. Their ability to control wave propagation has led to applications in a variety of fields, including optics, electromagnetics, acoustics, elasticity, and hydrodynamics. A particularly active research direction concerns spatio-temporal metamaterials, where the material properties are structured both in space and time \cite{Davies2025Roadmap,Garg2024TwoStep,WeZhWuLi2021}.

A basic mathematical model for wave propagation in such media is the wave equation with a space- and time-dependent coefficient
\begin{align}\label{wave}
\partial_{tt} u_\varepsilon
-
\operatorname{div}
\left(
\mu_\varepsilon(x,t)\nabla u_\varepsilon
\right)
=0,
\end{align}
subject to suitable initial and boundary conditions. The coefficient $\mu_\varepsilon(x,t)$ describes the material properties and may exhibit rapidly oscillating behavior in space and/or in time on a microscopic length scale $0<\varepsilon\ll1$. Such multiscale structures make direct numerical simulation challenging and motivate the derivation of effective models that describe the macroscopic wave dynamics through a wave equation with a space- and time-independent coefficient.

Homogenization theory provides a rigorous framework for deriving effective coefficients for highly oscillatory media in the asymptotic regime $\varepsilon\to0$. While spatial homogenization has been extensively studied for hyperbolic problems \cite{BrahimOtsmane1992Correctors,kristensson2003homogenization,Santosa1991Dispersive}, homogenization with time-modulated coefficients remains considerably less developed since the solutions can be unstable due to strong temporal oscillations \cite{colombini1984hyperbolic}. Formal homogenization of such time-modulated wave equations has been investigated in \cite{DV26,Touboul2024HighOrder}, while a rigorous homogenization analysis for space-time-modulated wave equations under weak temporal modulation is developed in \cite{CasadoDiaz2014Homogenization}. However, homogenization theory is inherently asymptotic: the homogenized coefficient approximates the heterogeneous medium only in the limit of vanishing microscopic scale. In practical applications, however, the microstructure is characterized by a fixed positive length scale $\varepsilon>0$, so that homogenized models generally approximate the heterogeneous dynamics only up to an error of order $\mathcal{O}(\varepsilon)$.

In many applications, the heterogeneous coefficient is often unknown, and only measurements of the propagating wave are available. This leads naturally to an inverse problem of determining effective material parameters directly from observations of solutions to \eqref{wave}. Recovering the full oscillatory coefficient $\mu_\varepsilon$ is generally ill-posed (cf.~\cite{Lions1978}), whereas identifying a constant effective material coefficient is feasible. Several approaches to this problem have been proposed for elliptic multiscale problems, including parameter estimation techniques \cite{Nolen2012}, separating oscillation methods \cite{Gulliksson_2016}, neural-network-based approaches \cite{PARK2022111420}, and optimization-based formulations introduced by Le Bris, Legoll, and Lemaire \cite{LeBrisLegollLemaire18,le2013approximation} and further developed in \cite{LeBrisLegollRuget26,ruget2025effective}. To the best of our knowledge, no corresponding framework has been developed for wave equations with space-time-dependent coefficients.

In this work, we extend the optimization-based framework to the wave equation \eqref{wave}. Given a solution of the heterogeneous problem, we identify an optimal constant coefficient such that the corresponding homogeneous wave equation reproduces the observed macroscopic dynamics. The coefficient is obtained by minimizing a cost functional measuring the residual of the effective equation. A key step in constructing the cost functional consists of applying the wave operator associated with the candidate coefficient and regularizing the residual with an inverse wave operator. This yields a quadratic cost functional, enabling efficient computation and rigorous analysis of the optimization problem.

We establish the asymptotic consistency of the proposed method by proving that, whenever a homogenized limit exists, the identified coefficient converges to the homogenized coefficient as $\varepsilon\to0$. We emphasize that, in periodic homogenization, the oscillatory coefficient generally does not converge to the effective coefficient either pointwise or weakly. Instead, convergence is obtained in the sense of $G$-convergence (or $H$-convergence; see \cite{tartar1997homogenization}). In contrast, our method directly determines a constant coefficient that converges pointwise to the homogenized coefficient.

Furthermore, we prove strong convergence in $L^2$ of the solution of the identified effective wave equation to the heterogeneous solution. In addition, we establish weak convergence of the first-order space and time derivatives in $L^2$.

We highlight that the proposed method is also applicable even in setting where homogenization is not available. Numerical experiments confirm our theoretical results and demonstrate that the proposed method recovers meaningful effective coefficients also for unstructured space-time-dependent coefficients, where classical homogenization techniques are not directly applicable. \\

\textbf{Outline.} In Section~\ref{sec:setting}, we introduce the mathematical framework for the identification of effective coefficients in the space-time-dependent wave equation. In Section~\ref{sec:optimization_approach}, we present the proposed optimization approach and establish the existence of minimizers for the associated cost functional. Section~\ref{sec:asymptotic_analysis} is devoted to the asymptotic analysis of the method under the assumption that a homogenized limit exists as $\varepsilon\rightarrow 0$. Finally, in Section~\ref{sec:numerical_results}, we present numerical experiments that validate the theoretical results and demonstrate the performance of the proposed approach for both homogenizable and more general heterogeneous coefficients.

\section{Time-varying wave equation} 
\label{sec:setting}
We consider the wave equation in the absence of sources with a space-time modulated coefficient
\begin{equation}  \label{wave-equation}
\begin{aligned}
	\partial_{tt} u_\varepsilon - \div ( \mu_\varepsilon(x,t) \nabla u_\varepsilon) & = 0 \quad && \text{in } \T \times (0,T) \\ 
    u_\varepsilon = v_0, \quad \partial_t u_\varepsilon & = v_1 \quad && \text{on } \T \times \{ t = 0\}.
\end{aligned}
\end{equation}
Here $\T = \R^d / \Z^d, d = 1,2,3$ is the $d$-dimensional torus, $(0,T)$ is a finite time interval, and $v_0$, $v_1$ are sufficiently smooth initial values. Considering the above wave equation on $\T$ means that we impose periodic boundary conditions in space. The coefficient $\mu_\varepsilon \in C^1([0,T],L^\infty(\T,\Rsym))$, where $\Rsym = \{ A \in \Rdd: A^\top = A \}$, and there exists a constant $\lambda>0$ such that  
\begin{align*}
    0 < \lambda |\xi|^2 \le \mu_\varepsilon(x,t) \xi \cdot \xi \quad \text{for all } \xi \in \R^d \text{ and almost all } (x,t) \in \T \times (0,T).
\end{align*}
For the initial data we assume throughout the work
\begin{align*}
    v_0 \in H^1(\T), \quad v_1 \in L^2(\T)
\end{align*}
so that the solution $u_\varepsilon$ of \eqref{wave-equation} satisfies the regularity properties (cf. \cite{RenardyRogers})
\begin{align} \label{regularity-u}
\begin{split}
    & u_\varepsilon \in L^2(0,T;H^1(\T)), \quad \partial_t u_\varepsilon \in L^2(0,T;L^2(\T)), \quad \partial_{tt} u_\varepsilon \in L^2(0,T;H^1(\T)^*), \\
    & \qquad \qquad \qquad \qquad u_\varepsilon \in C([0,T],H^1(\T)) \cap C^1([0,T],L^2(\T)).
\end{split}
\end{align}
Here we use the standard notation for Sobolev, Lebesgue and Bochner spaces. \\
In many physical applications, the explicit spatio-temporal dependence of the coefficient $\mu_\varepsilon$ is unknown. Instead, one can measure wave propagation through the modeled medium, leading to the solution $u_\varepsilon$ for prescribed initial data. Based on the knowledge of such a solution $u_\varepsilon$, we aim to recover an effective constant coefficient that describes the system governed by \eqref{wave-equation}. More precisely, our goal is to replace the spatio-temporally varying coefficient $\mu_\varepsilon$ with a constant symmetric positive semidefinite matrix $\mubar \in \A := \{ \mu \in \Rsym : 0 \le \mu \xi \cdot \xi \text{ for all } \xi \in \R^d \}$ such that the solution $u(\mubar)$ of the wave equation 
\begin{equation}  \label{mu-wave-equation}
\begin{aligned}
	\partial_{tt} u(\mubar) - \div (\mubar \nabla u(\mubar)) & = 0 \quad && \text{in } \T \times (0,T) \\ 
    u(\mubar) = v_0, \quad \partial_t u(\mubar) & = v_1 \quad && \text{on } \T \times \{ t = 0\}
\end{aligned}
\end{equation}
is a ``good'' approximation of $u_\varepsilon$. Note that $u(\mubar)$ is at least as smooth as $u_\varepsilon$ for any $\mubar \in \A$ and, in particular, satisfies \eqref{regularity-u}. The constant coefficient matrix is determined by optimizing a suitable cost functional, denoted by $\Phi_\varepsilon$. To rigorously assess the approximation properties of $\mubar$ and the corresponding solution $u(\mubar)$, we assume that there is a family of solutions $\{u_\varepsilon\}_{\varepsilon>0}$ which is ``rich'' enough and that converges, in an appropriate sense, to a homogenized limit satisfying a wave equation with a constant homogenized coefficient matrix. We specify our assumptions later in our analysis and present a particular case involving a periodic coefficient $\mu_\varepsilon$ for which the assumptions are fulfilled. Nevertheless, we emphasize that, from a computational perspective, the method is not restricted to settings in which these assumptions hold.

\section{Optimization approach}
\label{sec:optimization_approach}

The optimization approach considers the minimization problem
\begin{align} \label{minimization-problem}
	\inf_{\mu \in \A} \Phi_\varepsilon(\mu), \quad \A := \big\{ \mu \in \Rsym: \, \mu \xi \cdot \xi \ge 0 \, \, \forall \xi \in \R^d  \big\}
\end{align}
where $\Phi_\varepsilon: \A \rightarrow \R$ is an cost functional on the admissible set $\A$. The definition of the cost functional involves the solution operator
\begin{align}
    \Box^{-1}: L^2(0,T;H^1(\T)^*) \rightarrow L^2(0,T;H^1(\T))    
\end{align}
of the wave equation with unit speed and homogeneous initial data such that $v = \Box^{-1} f$ is given by the solution of
\begin{equation*}
\begin{aligned}
	\partial_{tt} v - \Delta v & = f \quad && \text{in } \T \times (0,T) \\ 
    v = 0, \quad \partial_t v & = 0 \quad && \text{on } \T \times \{ t = 0\}.
\end{aligned}
\end{equation*}
In particular, $\Box^{-1}$ is a well-defined and bounded operator, cf. \cite{Evans}. We define the cost functional as
\begin{align} \label{cost-function}
    \Phi_{\varepsilon}(\mu) :=  \| \Box^{-1}(  \partial_{tt} u_\varepsilon - \div( \mu \nabla u_\varepsilon ))\|^2_{L^2(0,T;L^2(\T))}.
\end{align}
Note that $\Phi_\varepsilon(\mu)$ is obtained by applying the operator $\Box^{-1}( \partial_{tt} - \div (\mu \nabla ))$ to the difference $u_\varepsilon - u(\mu)$ and measuring the error in $L^2(0,T;L^2(\T))$. The key idea to first apply the wave operator $\partial_{tt}-\operatorname{div}(\mu\nabla)$ to the defect $u_\varepsilon-u(\mu)$, and then map the result back to $L^2(0,T;L^2(\T))$ by means of the inverse wave operator yields an cost functional with an explicit dependence on the coefficient $\mu$. A similar strategy for elliptic problems was first proposed in \cite{LeBrisLegollLemaire18}, and we adopt a similar idea in the present setting. Furthermore, as we show below, the resulting cost functional is quadratic in $\mu$. Owing to this structure, the existence of minimizers can be established and the computation of minimizers, particularly by numerical methods, is straightforward. \\

We identify the finite dimensional space $\Rsym$ with $\R^{d_*}$ where
\begin{align*}
    d_* := \dim \Rsym = \frac{d(d+1)}{2}.
\end{align*}
For this purpose, we define an enumerating map
\begin{align} \label{enumeration}
    [\cdot,\cdot]: \{ (i,j): 1 \le i \le j \le d \} \rightarrow \{  1,\dots,d_* \}
\end{align}
that is bijective and so that $[i,j] = m$ enumerates the tuples of indices $(i,j)$, $1 \le i \le j \le d$. Then, we identify any matrix $\mu \in \Rsym$ with the vector $\bfmu \in \R^{d_*}$ via
\begin{align} \label{identification}
    \bfmu_{[i,j]} = \mu_{ij}, \quad 1 \le i \le j \le d.
\end{align}
A straightforward computation shows that the cost functional can be written as the quadratic form on $\mathbb{R}^{d_*}$,
\begin{align} \label{quadratic_function}
\Phi_\varepsilon(\mu) = \bfmu^\top A \bfmu -2 b^\top \bfmu + c,
\end{align}
where $c = \| \Box^{-1} \partial_{tt} u_\varepsilon \|_{L^2(0,T;L^2(\T))}^2$ and $A \in \R^{d_* \times d_*}$ and $b \in \R^{d_*}$ are given by
\begin{align} \label{eq:coeff}
    A_{[i,j],[k,\ell]} = (\alpha_{ij}, \alpha_{k\ell})_{L^2(0,T,L^2(\T))}, \quad b_{[i,j]} = (\Box^{-1} \partial_{tt} u_\varepsilon, \alpha_{ij})_{L^2(0,T,L^2(\T))}
\end{align}
for $1 \le i \le j < d$, $1 \le k \le \ell \le d$ and with
\begin{align*}
    \alpha_{ij} = \begin{cases} \Box^{-1} \partial_{ii} u_\varepsilon, & 1 \le i = j \le d \\
    2 \, \Box^{-1} \partial_{ij} u_\varepsilon, & 1 \le i < j \le d \end{cases}.
\end{align*}
In one spatial dimension ($d=1$), the leading-order coefficient in
\eqref{quadratic_function} reduces to
\begin{align*}
    \|\Box^{-1}\partial_{xx}u_\varepsilon\|_{L^2(0,T;L^2(\T))}^2,
\end{align*}
which is, without loss of generality, strictly positive. Consequently, the cost functional
$\Phi_\varepsilon$ is strictly convex and therefore attains a unique
minimizer on the admissible set $\A$, namely
    \begin{align} \label{minimum}
        \mubar = \max \Big\{ 0, \frac{\int_0^T (\Box^{-1} \partial_{tt} u_\varepsilon,  \Box^{-1} \partial_{xx} u_\varepsilon)_{L^2(\T)} dt}{\|   \Box^{-1} \partial_{xx} u_\varepsilon \|^2_{L^2(0,T;L^2(\T))}} \Big\}.
    \end{align}
In higher spatial dimensions, the situation is considerably more subtle.
The existence and uniqueness of minimizers over the space
$\Rsym$ are completely determined by the algebraic properties of the
matrix $A$ and the vector $b$. More precisely, if $A$ is positive
semidefinite and $b\in\operatorname{ran}(A)$, then
$\Phi_\varepsilon$ attains a minimizer over $\Rsym$. If, in addition, $A$ is positive definite the minimizer is unique and given by
\begin{align} \label{minimum_if_pos}
\mubar = A^{-1}b.
\end{align}
However, there is no a priori guarantee that this minimizer belongs to the admissible set $\A$ of symmetric positive-definite
matrices.

From a computational point of view, this is not a serious obstacle.
Once the coefficients in \eqref{quadratic_function} have been assembled,
it is straightforward to verify whether $A$ is positive definite and
whether the corresponding minimizer belongs to $\A$. Moreover, these
computations are cheap since the the system \eqref{quadratic_function} is
low dimensional. On the other hand, establishing the required algebraic
properties of $A$ by means of an a priori analytical argument
appears to be considerably more challenging.

Nevertheless, it is possible to establish the existence of minimizers
directly from the abstract cost functional
\eqref{cost-function}. The key requirement is that the measured solution
$u_\varepsilon$ is sufficiently rich, in the sense that the wave equation
\eqref{wave-equation} propagates a suitable family of Fourier
modes. This property is guaranteed, for instance, if the initial data
contain such a collection of Fourier modes. We formulate this
requirement in the following assumption, which seems to be a crucial  ingredient for proving the existence of minimizers.

\begin{assumption}[Richness of the initial data] \label{assumption2}
    Let
    \begin{align*}
      v_0(x) = \sum_{k \in \Z^d} c_k  e^{2 \pi \i k \cdot x}, 
    \end{align*}
    be the Fourier series of the initial data $v_0 \in H^1(\T)$. We assume that there is a collection of wave vectors $k_j \in \Z^d$, $j = 1,\dots, d_*$ such that $c_{k_j} \neq 0$ and
    \begin{align*}
        \mathrm{span}\big\{ k_j k_j^\top : j = 1, \dots, d_*\big\} = \Rsym.
    \end{align*}
\end{assumption}

Assumption \ref{assumption2} may appear rather abstract at first sight. We first note that it is automatically satisfied in one spatial dimension ($d=1$), provided that the initial datum $v_0$, and hence the measured solution $u_\varepsilon$, is not identically zero. 

In higher dimensions, Assumption \ref{assumption2} can be realized by combining at most $d_*$ measurements, each corresponding to the propagation of a single Fourier mode under the wave equation \eqref{wave-equation}. Owing to the linearity of the wave equation, the resulting solution can subsequently be obtained by superposition of the
propagated Fourier modes. As an example, consider the two-dimensional case ($d=2$), for
which $d_*=3$. A suitable choice of wave vectors is
\begin{align*}
k_1=
\begin{pmatrix}
1\\
0
\end{pmatrix},
\qquad
k_2=
\begin{pmatrix}
0\\
1
\end{pmatrix},
\qquad
k_3=
\begin{pmatrix}
1\\
1
\end{pmatrix}.
\end{align*}
These vectors satisfy $\mathrm{span}\{k_1k_1^\top,k_2k_2^\top,k_3k_3^\top\} = \R^{2 \times 2}_{\mathrm{sym}}$ and therefore Assumption \ref{assumption2} is fulfilled by measuring the propagation of these three Fourier modes through the heterogeneous medium. \\

Using Assumption \ref{assumption2} we now prove existence of a minimizer of \eqref{minimization-problem}. Note, however, that the minimizer does not have to be unique.

\begin{lemma} \label{lem:existence}
    Suppose Assumption \ref{assumption2} holds. Then, it holds:
    \begin{itemize}
        \item [(i)] $\Phi_\es: \A \rightarrow \R$ from \eqref{cost-function} is well-defined and continuous.
        \item[(ii)] There are $C, \alpha > 0 $ such that $\Phi_\es(\mu) \ge \alpha |\mu|^2 - C$ for all $\mu \in \A$.
        \item[(iii)] $\Phi_\es$ attains a  minimum
        \begin{align*}
            \mubar = \argmin_{\mu \in \A} \Phi_\varepsilon(\mu).
        \end{align*}
    \end{itemize}
\end{lemma}
\begin{proof} (i) As the regularity \eqref{regularity-u} holds, the cost functional $\Phi_\varepsilon$ is well-defined. We note that the operator $\Box^{-1}$ is bounded, i.e.,
        \begin{align*}
            \| \Box^{-1} f \|_{L^2(0,T;L^2(\Omega))} \le C \| f  \|_{L^2(0,T;H^1(\Omega)^*)}.
        \end{align*}
        Then, we have for $\mu_1,\mu_2 \in \A$ and with $P_\varepsilon(\mu) := \partial_{tt}u_\es-\div(\mu\nabla u_\es)$,
        \begin{align*}
        \begin{split}
            |\Phi_\es(\mu_1)-\Phi_\es(\mu_2)|&=\Big|\|\Bx P_\es (\mu_1)\|_{L^2(0,T;L^2(\T))}^2-\|\Bx P_\es (\mu_2)\|_{L^2(0,T;L^2(\T))}^2\Big| \\
            &\leq \|\Bx(P_\es (\mu_1)-P_\es (\mu_2))\|_{L^2(0,T;L^2(\T))}^2 \\
            &\leq C\|P_\es (\mu_1)-P_\es (\mu_2)\|_{L^2(0,T;H^1(\T)^*)}^2 \\
            &=C \|\div((\mu_1-\mu_2)\nabla u_\es)\|_{L^2(0,T;H^1(\T)^*)}^2.
        \end{split}
        \end{align*}
        Now, for any $\phi\in L^2(0,T;H^1(\T))$, we have
        \begin{align*}
            \int_0^T\langle\div((\mu_1-\mu_2)\nabla u_\es),\phi\rangle dt&=-\int_0^T\int_{\T}(\mu_1-\mu_2)\nabla u_\es\nabla\phi\,dxdt\\
            &\leq \|\mu_1-\mu_2\|\|\nabla u_\es\|_{L^2(0,T;L^2(\T))}\|\nabla \phi\|_{L^2(0,T;L^2(\T))}.
        \end{align*}
        This implies
        \begin{align*}
            \|\div((\mu_1-\mu_2)\nabla u_\es)\|_{L^2(0,T;H^1(\T)^*)}\leq |\mu_1-\mu_2 | \, \|\nabla u_\es\|_{L^2(0,T;L^2(\T))}
        \end{align*}
        which proves the continuity of $\Phi_\varepsilon$. \\

        (ii)  Let $Q_\es(\mu):=\Bx(\div(\mu\nabla u_\es))$ for any $\mu\in\A$. Clearly $Q_\es$ is a linear operator. We first show that $Q_\es$ is a one-to-one mapping. Let $Q_\es(\mu_0)=0$ for some $\mu_0\in\A$, then 
        \begin{align*}
            \Bx(\div(\mu_0\nabla u_\es))=0.
        \end{align*}
        Since $\Bx$ is one-to-one mapping, we conclude
        \begin{align} \label{eq1:proof_coercivity}
            \div(\mu_0\nabla u_\es)=0 \quad \text{in } L^2(0,T;H^1(\T)^*).
        \end{align}
        We have $u_\es \in C([0,T], H^1(\T))$, cf. \eqref{regularity-u}, then the equation \eqref{eq1:proof_coercivity} holds pointwise for all  $t \in [0,T]$ and in $H^1(\T)^*$. In particular, we can set $t = 0$ and obtain
        \begin{align} \label{eq2:proof_coercivity}
            \div(\mu_0\nabla v_0 ) = 0 \quad \text{in } H^1(\T)^*.
        \end{align}
        Plugging in the Fourier series from Assumption \ref{assumption2} into \eqref{eq2:proof_coercivity}, yields
        \begin{align*}
            k_j^\top \mu_0 k_j = 0, \quad j = 1,\dots,d_*.
        \end{align*}
        By Lemma \ref{lem:appendix_uniqueness} we conclude $\mu_0 = 0$, which shows that $Q_\es$ is a one-to-one mapping. Since $\A$ is a finite-dimensional space and the linear map $Q_\es$ defined on $\A$ is one-to-one, there exist a $C>0$, such that
        \begin{align*}
           |\mu| \leq C \|Q_\es(\mu)\|_{L^2(0,T;L^2(\T))}.
        \end{align*}
        It follows,
        \begin{align*}
   \sqrt{\Phi_\es(\mu)}&=\|\Bx\partial_{tt}u_\es-\Bx (\div(\mu \nabla u_\es))|_{L^2(0,T;L^2(\T))}\\
   &\geq \|\Bx (\div(\mu \nabla u_\es))|_{L^2(0,T;L^2(\T))}-\|\Bx\partial_{tt}u_\es\|_{L^2(0,T;L^2(\T))}\\
   &\geq C | \mu |-\|\Bx\partial_{tt}u_\es\|_{L^2(0,T;L^2(\T))}.
        \end{align*}

        (iii) follows from (i) and (ii).
\end{proof}

\section{Asymptotic consistency analysis}
\label{sec:asymptotic_analysis}

The optimization approach proposed in Section~\ref{sec:optimization_approach} should be consistent with homogenization theory in the sense that the identified effective coefficient recovers the homogenized coefficient whenever the latter exists. Although homogenization theory for wave equations with space-time-dependent heterogeneous coefficients is considerably less developed than its counterpart for purely spatially varying media, rigorous homogenization results are available in certain settings, see \cite{CasadoDiaz2014Homogenization}, while in other cases the existence of an effective limit is supported by formal asymptotic arguments, see \cite{DV26}.

In this section, we establish the asymptotic consistency of the proposed identification approach. More precisely, we show that the optimal coefficient $\mubar$ converges to the homogenized coefficient $\mu_{\hom}$ as the microscopic length scale vanishes, $\varepsilon \rightarrow 0$. Moreover, we show convergence of the solution $u(\mubar)$ of the identified effective model towards the homogenized solution $u_{\hom}$ strongly in $L^2(0,T;L^2(\T))$ and weak convergence for the first space and time derivatives. Rather than relying on a specific homogenization result, our analysis is formulated under an abstract assumption on the existence of a homogenized limit.

\begin{assumption} \label{assumption1}
Let $\{\mu_\varepsilon\}_{\varepsilon>0}$ be a family of coefficients such that the associated family of solutions $\{u_\varepsilon\}_{\varepsilon>0}$ to \eqref{wave-equation} satisfies \eqref{regularity-u}. We assume that there exists a function $u_\hom \in L^2(0,T;H^1(\T))$ with $\partial_t u_{\hom}\in L^2(0,T;L^2(\T))$ such that
\begin{align}\label{eq:weak_convergence_of_nabla_u_epsilon}
\nabla u_\varepsilon \rightharpoonup \nabla u_{\hom},
\qquad
\partial_t u_\varepsilon \rightharpoonup \partial_t u_{\hom}
\end{align}
weakly in $L^2(0,T;L^2(\T))$ as $\varepsilon \to 0$ and $u_\hom$ satisfies a homogenized wave equation
\begin{equation}  \label{hom-equation}
\begin{aligned}
	\partial_{tt} u_{\hom} - \div( \mu_{\hom} \nabla u_{\hom}) & = 0 \quad && \text{in } \T \times (0,T) \\ 
    u_{\hom} = v_0, \quad \partial_t u_{\hom} & = v_1 && \text{on } \T \times \{ t = 0\}
\end{aligned}
\end{equation}
for some positive definite $\mu_{\hom}\in\Rsym$.
\end{assumption}

We denote by $Y$ the unit cube $(0,1)^d \subset \R^d$. Then, an example coefficient family for which Assumption \ref{assumption1} is satisfied is given by
\begin{align} \label{ass:periodic}
    \mu_\varepsilon(x,t) = \mu_0(x/\varepsilon) + \varepsilon \mu_1(t/\varepsilon), \quad \varepsilon > 0,
\end{align}
where $\mu_0 \in L^\infty(Y,\Rspd)$ is periodic in $Y$ and $\mu_1 \in C^1((0,1);\Rspd)$ is $1$-periodic.
In this case, we obtain a uniform energy bound for $u_\es$ which implies the weak convergence \eqref{eq:weak_convergence_of_nabla_u_epsilon} (see \cite{CasadoDiaz2014Homogenization}). The limit $u_{\hom}$ satisfies the homogenized equation \eqref{hom-equation}, 
with the homogenized coefficient $\mu_{\hom} \in \Rspd$ is given by
\begin{align*}
    \big( \mu_{\hom} \big)_{ij} = \int_Y \mu_0(y)(e_i + \nabla w_i ) \cdot (e_j + \nabla w_j ) \, dy, \quad i,j = 1,\dots, d
\end{align*}
where $e_i \in \R^d$ is the $i$-th canonical unit vector and $w_i \in H^1(Y)$ is the unique periodic solution to the cell-problem
\begin{align*}
    - \div ( \mu_0(y) (e_i + \nabla w_i ) ) = 0 \quad \text{in } Y, \quad \int_Y w_i(y) \, dy = 0.
\end{align*}
We note that for $d = 1$ the homogenized coefficient is simply given by the harmonic mean
\begin{align*}
    \mu_\hom = \Big( \int_0^1 \mu_0(y)^{-1} \, dy \Big)^{-1}.
\end{align*}

Further, we emphasize that Assumption \ref{assumption1} does not require precise knowledge of the homogenized coefficient $\mu_\hom$ and only assumes its existence. G-convergence or H-convergence to a homogenized coefficient, as in \cite{tartar1997homogenization}, for which an explicit formula is typically not available, is sufficient. \\
Finally, we note that if $u_\hom$ from Assumption \ref{assumption1} exists, then it has the regularity
\begin{align*}
    u_{\hom} \in L^2(0,T;H^1(\T)), \quad \partial_t u_{\hom} \in L^2(0,T;L^2(\T)), \quad \partial_{tt} u_{\hom} \in L^2(0,T;H^1(\T)^*).
\end{align*}

\subsection{Convergence of the minimizing coefficient}

In this section, we aim to prove the first main result regarding consistency, namely the convergence of the minimizing coefficient towards the homogenized coefficient. As a first step, we show that the cost functional converges to zero.

\begin{theorem}\label{Theorem:consistancy}
Suppose that Assumption \ref{assumption1} holds. Then,
\begin{align*}
    \inf_{\mu \in \A} \Phi_\varepsilon(\mu) \rightarrow 0 \quad \text{as} \quad \varepsilon \rightarrow 0.
\end{align*}

\end{theorem}

\begin{proof}

First, we show that $\Phi_\varepsilon(\mu_\hom)$ is uniformly bounded. Let $v_\varepsilon := \Box^{-1} \big(\partial_{tt} u_{\varepsilon} - \div (\mu_{\hom} \nabla u_{\varepsilon}) \big)$. Then, using \eqref{wave-equation}, we obtain that $v_\varepsilon$ solves 
\begin{equation*}
\begin{aligned}
	\partial_{tt} v_\varepsilon - \Delta v_\varepsilon & = \div \big( (\mu_{\hom} - \mu_\varepsilon ) \nabla u_\varepsilon \big) \quad && \text{in } \T \times (0,T), \\ 
    v_\varepsilon & = 0, \quad \partial_t v_\varepsilon = 0 && \text{on } \T \times \{ t = 0\}.
\end{aligned}
\end{equation*}
Therefore, for any test function $\phi \in H^1(\T)$ and all $t \in (0,T)$
\begin{align} \label{weak-form}
    \langle \partial_{tt}v_\varepsilon(t),\phi \rangle +\int_{\T}\nabla v_\es(t) \cdot \nabla \phi \, dx = \langle \div \big( (\mu_{\hom}-\mu_\es(t))\nabla u_\es(t) \big), \phi\rangle.
\end{align}
In the following we choose for $t \in (0,T)$ the test function $\phi(t) = - \Delta^{-1} (\partial_t v_\varepsilon(t))$, i.e., $\phi(t) \in H^1(\T)$ is the unique solution to 
\begin{align*}
    -\Delta \phi(t) =\partial_t v_\es(t) \quad \mbox{in } \T, \quad \int_{\T} \phi(t) \, dx = 0.
\end{align*}
Then, we have $\| \nabla \phi(t) \|_{L^2}^2 = \langle \partial_t v_\varepsilon(t), \phi(t) \rangle$. Furthermore, we have the identities
\begin{align*}
    \langle\partial_{tt}v_\varepsilon(t),\phi(t) \rangle&=\frac{1}{2}\frac{d}{dt}\| \nabla \phi(t) \|_{L^2(\T)}^2,\\
    \int_{\T}\nabla v_\es(t) \cdot \nabla \phi(t) \, dx&=\frac{1}{2}\frac{d}{dt} \| v_\es(t) \|_{L^2(\T)}^2, \\
    \langle\div \big( (\mu_{\hom}-\mu_\es(t))\nabla u_\es(t) \big), \phi(t) \rangle&=-\int_{\T}(\mu_{\hom}-\mu_\es(t))\nabla u_\es(t) \cdot \nabla \phi(t) \,dx.
\end{align*}
We obtain from \eqref{weak-form} with $\phi = - \Delta^{-1}(\partial_t v_\varepsilon (t) )$
\begin{align*}
    \frac{1}{2}\frac{d}{dt}\Big(\| \nabla \phi(t) \|_{L^2(\T)}^2+\| v_\es(t)\|_{L^2(\T)}^2\Big) & =-\int_{\T}(\mu_{\hom}-\mu_\es(t))\nabla u_\es(t) \cdot  \nabla\phi(t) \,dx\\
    &\leq \|\mu_\es(t) -\mu_{\hom}\|_{L^\infty(\T)}\|\nabla u_\es(t) \|_{L^2(\T)}\|\nabla \phi(t) \|_{L^2(\T)}.
\end{align*}
Let $p_\es(t) := \| \nabla \phi(t) \|_{L^2(\T)}^2 + \| v_\es(t)\|_{L^2(\T)}^2$ and since $\mu_\es$ is bounded, we have
\begin{align*}
     \frac{1}{2}\frac{d}{dt} p_\es(t) \leq C \|\nabla u_\es(t) \|_{L^2(\T)}\sqrt{p_\es(t)}.
\end{align*}
Therefore, 
\begin{align*}
     \frac{d}{dt}(\sqrt{p_\es(t)})\leq C \|\nabla u_\es(t) \|_{L^2(\T)}.
\end{align*}
Integrating the aforementioned inequality from $0$ to $T$ w.r.t.~$t$, using the fact that $p_\es(0) = 0$ and the uniform bound for $\|\nabla u_\es \|_{L^2(0,T;L^2(\T))}$ implied by Assumption \ref{assumption1}, we conclude
\begin{align*}
    \Phi_\varepsilon(\mu_\hom) = \|v_\es\|_{L^2(0,T;L^2(\T))}^2 \leq p_\es(T) \leq C.
\end{align*}
Since $\Phi_{\varepsilon}(\mu_{\hom})$ is uniformly bounded, we can extract a subsequence such that $\Phi_{\varepsilon}(\mu_{\hom}) \rightarrow \Phi_{\star}$ as $\varepsilon \rightarrow 0$. We show that $\Phi_\star = 0$. By the definition of $v_\varepsilon$ and the weak convergence in Assumption \ref{assumption1} we infer for an arbitrary test function $\phi \in C^\infty((0,T) \times \T)$ with $\phi(T,\cdot) = 0$ that
\begin{align*}
    & \int_0^T \int_{\T} \Box v_\varepsilon \, \phi \, dxdt = \int_0^T \int_{\T} (\partial_{tt} u_\varepsilon - \div( \mu_\hom \nabla u_\varepsilon )) \, \phi \, dxdt \\
    & = - \int_0^T \int_{\T} \partial_t u_\varepsilon \, \partial_t \phi \, dxdt + \int_0^T \int_{\T} \mu_\hom \nabla u_\varepsilon \cdot \nabla \phi \, dxdt + \int_0^T \int_{\T} v_1 \, \phi(0) \, dxdt \\
    & \rightarrow - \int_0^T \int_{\T} \partial_t u_\hom \, \partial_t \phi \, dxdt + \int_0^T \int_{\T} \mu_\hom \nabla u_\hom \cdot \nabla \phi \, dxdt + \int_0^T \int_{\T} v_1 \, \phi(0) \, dxdt \\
    & = 0, \quad \text{as } \varepsilon \rightarrow 0,
\end{align*}
where we used that $u_\hom$ solves \eqref{hom-equation}. Therefore,
\begin{align} \label{eq:weak_conv}
    \Box v_\varepsilon \rightharpoonup 0 \quad \text{ in } L^2(0,T;L^2(\T)).
\end{align}
Next, we define $\tilde{\Box}^{-1}: L^2(0,T;H^1(\T)^*) \rightarrow  L^2(0,T;H^1(\T))$ to be the solution operator of the wave equation with zero final time condition, i.e., $w = \tilde{\Box}^{-1} f$ is the unique solution to
 \begin{equation*} 
\begin{aligned}
	\partial_{tt} w - \Delta w & = f \quad && \text{in } \T \times (0,T) \\ 
    w = 0, \quad \partial_t w & = 0 && \text{on } \T \times \{ t = T\}.
\end{aligned}
\end{equation*}
Then, $\tilde{\Box}^{-1}$ is well-defined and the identity $\Box \tilde{\Box}^{-1} f = f$ holds in $L^2(0,T;H^1(\T)^*)$. With $w_\varepsilon := \tilde{\Box}^{-1} v_\varepsilon \in L^2(0,T;H^1(\T))$ we conclude using \eqref{eq:weak_conv},
\begin{align*}
    \Phi_\varepsilon(\mu_\hom) & = \int_0^T \int_{\T} |v_\varepsilon|^2 \, dxdt = \int_0^T \int_{\T} v_\varepsilon \, \Box w_\varepsilon \, dxdt \\
    & = \int_0^T \int_{\T} v_\varepsilon \, \partial_{tt} w_\varepsilon \, dxdt + \int_0^T \int_{\T} \nabla v_\varepsilon \cdot \nabla w_\varepsilon \, dxdt \\
    & = - \int_0^T \int_{\T} \partial_t v_\varepsilon \, \partial_{t} w_\varepsilon \, dxdt + \int_{\T} \big( v_\varepsilon(T)  \underbrace{\partial_t w_\varepsilon (T)}_{=0}  - \underbrace{v_\varepsilon(0)}_{= 0}\partial_t w_\varepsilon (0) \big) dx \\
    & \quad +  \int_0^T \int_{\T} \nabla v_\varepsilon \cdot \nabla w_\varepsilon \, dxdt \\
    & = \int_0^T \int_{\T} \partial_{tt} v_\varepsilon \, w_\varepsilon \, dxdt + \int_{\T} \big( \partial_t v_\varepsilon(T)  \underbrace{\partial_t w_\varepsilon (T)}_{=0}  - \underbrace{\partial_t v_\varepsilon(0)}_{= 0}\partial_t w_\varepsilon (0) \big) dx \\
    & \quad +  \int_0^T \int_{\T} \nabla v_\varepsilon \cdot \nabla w_\varepsilon \, dxdt \\
    & = \int_0^T \int_{\T} (\partial_{tt} v_\varepsilon - \Delta v_\varepsilon) w_\varepsilon \, dxdt = \int_0^T \int_{\T} \Box v_\varepsilon \, w_\varepsilon \, dxdt \rightarrow 0 \quad \text{as } \varepsilon \rightarrow 0.
\end{align*}
This shows $\Phi_\star = 0$ and it follows
\begin{align*}
    0 \le \inf_{\mu \in \A} \Phi_\varepsilon(\mu) \le \Phi_{\varepsilon}(\mu_{\hom}) \rightarrow \Phi_{\star} = 0 \quad \text{as } \varepsilon \rightarrow 0.
\end{align*}
\end{proof}

As a next step, we show that if $u_{\hom}$ from Assumption \ref{assumption1} solves a wave equation with constant coefficient, the coefficient already needs to be $\mu_{\hom}$. In other words, this means that the constant coefficient in the wave equation is uniquely determined from the knowledge of the solution.

\begin{proposition}\label{Lemma:uniqueness_of_mu}
    Suppose Assumption \ref{assumption2}-\ref{assumption1} and let
    \begin{align*}
        F: \Rsym \rightarrow L^2(0,T;H^1(\T)^*), \quad \mu \mapsto \partial_{tt} u_\hom - \div(\mu \nabla u_\hom).
    \end{align*}
    Then $F$ has a unique root at $\mu = \mu_\hom$, i.e., $F(\mu_\hom) = 0$. 
\end{proposition}

\begin{proof}
    By Assumption \ref{assumption1} we have $F(\mu_\hom) = 0$, so it remains to show that $\mu_\hom$ is the only root of $F$. We consider the Fourier series for $u_{\hom}(\cdot,t) \in L^2(\T)$ given by
    \begin{align*}
        u_\hom(x,t) = \sum_{k \in \Z^d} \hat{u}_k(t) e^{2\pi \i k \cdot x}.
    \end{align*}
    For any $k_j \in \Z^d$, $j = 1,\dots,d_*$ from Assumption \ref{assumption2} we have, by \eqref{hom-equation}, that $\hat u_{k_j} \in C^2(0,T;\R)$ is the unique solution to
    \begin{align*}
        \hat{u}_{k_j} '' + (k_j^\top \mu_\hom k_j) \hat u_{k_j} = 0, \quad \hat{u}_{k_j}(0) = c_{k_j}, \quad \hat{u}_{k_j}'(0) = b_{k_j}.
    \end{align*}
    Here $c_{k_j}$ is the Fourier coefficient from Assumption \ref{assumption2} and $b_{k_j}$ is the corresponding Fourier coefficient of $v_1$, i.e., $v_1(x) = \sum_{k \in \Z^d} b_k e^{2\pi \i k\cdot x}$. Since $c_{k_j} \neq 0$, by Assumption \ref{assumption2}, there is $t_\star \in (0,T)$ such that $\hat u_{k_j}(t_\star) \neq 0$. We conclude
    \begin{align*}
        0 < k_j^\top \mu_\hom k_j = -\frac{\hat{u}_{k_j}''(t_\star)}{\hat{u}_{k_j}(t_\star)}, \quad j = 1,\dots, d_\star.
    \end{align*}
    Now let $\mu \in \Rsym$ be another root of $F$, i.e., $\partial_{tt} u_\hom - \div( \mu \nabla u_\hom) = F(\mu) = 0$. Then, using the Fourier series if $u_\hom$, we obtain
    \begin{align*}
        \hat{u}_{k_j} '' + (k_j^\top \mu k_j) \hat u_{k_j} = 0, \quad \hat{u}_{k_j}(0) = c_{k_j}, \quad \hat{u}_{k_j}'(0) = b_{k_j}
    \end{align*}
    and
    \begin{align*}
         k_j^\top \mu k_j = -\frac{\hat{u}_{k_j}''(t_\star)}{\hat{u}_{k_j}(t_\star)}, \quad j = 1,\dots,d_\star.
    \end{align*}
    We conclude
    \begin{align*}
        k_j^\top (\mu - \mu_{\hom}) k_j = 0, \quad j = 1,\dots,d_\star.
    \end{align*}
    Now Lemma \ref{lem:appendix_uniqueness} implies $\mu = \mu_{\hom}$ which proves the claim.
\end{proof}

We are now in the position to show the convergence of the minimizing coefficient.

\begin{theorem} \label{theorem:convergence_parameter}
    Suppose that Assumption \ref{assumption2}-\ref{assumption1} hold. Then,
    \begin{align*}
    \mubar \rightarrow \mu_{\hom} \quad \text{as} \quad \varepsilon \rightarrow 0.
\end{align*}
\end{theorem}
\begin{proof}
    By Theorem \ref{Theorem:consistancy}, we have $\Phi_\es(\mubar)\rightarrow 0$ as $\es\rightarrow 0$. Therefore,  $v_\es\rightarrow 0$ strongly in $L^2(0,T;L^2(\T))$ where $v_\es=\Box^{-1} \big(\partial_{tt} u_{\varepsilon} - \div( \mubar \nabla u_{\varepsilon})\big)$. Let $\phi\in C_c^\infty((0,T)\times\T)$. Using the weak form for $v_\es$, we have
    \begin{align*}
    \int_0^T\langle\partial_{tt}v_\varepsilon,\phi\rangle\,dt +\int_0^T\int_{\T}\nabla v_\es \cdot \nabla \phi \, dx\,dt=\int_0^T\langle \partial_{tt} u_{\varepsilon} - \div ( \mubar \nabla u_{\varepsilon} ), \phi\rangle\,dt.
\end{align*}
Using integration by parts and the Cauchy-Schwarz inequality, we estimate
\begin{align*}
  \int_0^T\int_{\T}u_\es\partial_{tt}\phi \,dx\,dt+ \int_0^T\int_{\T} \mubar \nabla u_\es \cdot \nabla\phi \,dx\,dt&= \int_0^T\langle \partial_{tt} u_{\varepsilon} - \div ( \mubar \nabla u_{\varepsilon}), \phi\rangle\,dt\\
  &=\int_0^T\int_{\T}v_\es (\partial_{tt}
\phi-\Delta \phi)\,dx\,dt\\
&\leq \|v_\es\|_{L^2(0,T;L^2(\T))}\|(\partial_{tt}
\phi-\Delta \phi)\|_{L^2(0,T;L^2(\T))}.
\end{align*}
Since $v_\es\rightarrow 0$ strongly in ${L^2(0,T;L^2(\T)}$ as $\es\rightarrow0$, we obtain
\begin{align}\label{eq:limit_uepsilon_weakform}
  \int_0^T\int_{\T}u_\es\partial_{tt}\phi \,dx\,dt+ \int_0^T\int_{\T} \mubar \nabla u_\es \cdot \nabla\phi \,dx\,dt \rightarrow 0 \quad \text{as } \es \rightarrow 0.
\end{align}
Using the weak coercivity of $\Phi_\varepsilon$ from Theorem \ref{lem:existence} (ii) and the convergence from Theorem \ref{Theorem:consistancy} it follows
\begin{align*}
    \alpha |\mubar| - C \le \Phi_\varepsilon(\mubar) \rightarrow 0 \quad \text{as } \varepsilon \rightarrow 0.
\end{align*}
Therefore, $\mubar$ is uniformly bounded.
Now we can extract any converging subsequence (again denoted by $\mubar$) for which we have $\mubar\rightarrow \mu_*\in \R$. We now show that $\mu_*=\mu_{\hom}$. Again, using that $u_\es$ converges strongly to $u_\hom$ in $L^2(0,T;L^2(\T))$ and \eqref{eq:limit_uepsilon_weakform}, we have
\begin{align*}
    0&=\lim_{\es\rightarrow 0} \int_0^T\int_{\T}u_\es\partial_{tt}\phi \,dx\,dt+  \lim_{\es\rightarrow 0} \int_0^T\int_{\T} \mubar \nabla u_\es \cdot \nabla\phi \,dx\,dt\\
    &=\int_0^T\int_{\T}u_{\hom}\partial_{tt}\phi \,dx\,dt +  \int_0^T\int_{\T} \mu_* \nabla u_{\hom} \cdot  \nabla \phi \,dy \,dx\,dt.
\end{align*}
Thus, it holds in $L^2(0,T;H^1(\T)^*)$ that
    \begin{align*}
        \partial_{tt} u_{\hom} - \div(\mu_* \nabla u_{\hom}) = 0 \quad \text{in } \T \times (0,T).
    \end{align*}
Now, Proposition \ref{Lemma:uniqueness_of_mu} implies $\mu_* = \mu_\hom$ which proves the claim.
\end{proof}

\subsection{Convergence of the solution}

In this section, we aim to prove our second main result regarding consistency, namely the convergence of the solution associated with the minimizing coefficient towards the homogenized solution. We will then conclude that the solution is also a good approximation of the given solution $u_\es$. Naturally, the convergence of $\mubar$ to the homogenized coefficient $\mu_\hom$ will be a key ingredient. We first conclude uniform bounds for the space and time derivatives of $u(\mubar)$ via standard energy estimates.

\begin{proposition} \label{lem:uniform_bound}
Suppose that Assumption \ref{assumption2}-\ref{assumption1} hold. Then, there is $C > 0$ independent of $\varepsilon$ such that
\begin{align*}
    \| \partial_t u(\mubar) \|_{L^\infty(0,T;L^2(\T))} + \| \nabla u(\mubar) \|_{L^\infty(0,T;L^2(\T))} \le C.
\end{align*}
\end{proposition}

\begin{proof}
Since $\mubar \rightarrow \mu_\hom$ as $\varepsilon \rightarrow 0$ and since $\mu_\hom$ is positive definite by Assumption \ref{assumption1}, there are constants $\varepsilon_0, \alpha,\beta > 0$ independent of $\varepsilon$ such that for all $0 < \varepsilon \le \varepsilon_0$
\begin{align*}
    0 < \alpha |\xi|^2 \le \mubar \xi \cdot \xi \le \beta |\xi|^2 \quad \text{for all } \xi \in \R^d.
\end{align*}
Standard energy estimates for the wave equation then yield uniformly on $(0,T)$ and for $0 < \varepsilon \le \varepsilon_0$
\begin{align*}
    \| \partial_t u(\mubar) \|_{L^2(\T)}^2 + \alpha \| \nabla u(\mubar) \|_{L^2(\T)}^2 \le \| v_1 \|_{L^2(\T)}^2 + \beta \| v_0 \|_{H^1(\T)}^2.
\end{align*}
This proves the claim.
\end{proof}

We are now in the position to prove our second main result, which is the $L^2$-convergence of  $u(\mubar)$ to the given solution $u_\es$ from Assumption \ref{assumption1}.

\begin{theorem}\label{theorem:strong_convergence_of_u_epsilon}
Suppose that Assumption \ref{assumption2}-\ref{assumption1} hold. Then for $\es\rightarrow 0$, we have
\begin{align*}
     u(\mubar)  \rightarrow u_\hom \quad \text{strongly in } L^2(0,T;L^2(\T)).
\end{align*}
Moreover, 
\begin{align*}
   u_\varepsilon - u(\mubar) \rightarrow 0 \quad \text{strongly in } L^2(0,T;L^2(\T)).
\end{align*} 

\end{theorem}

\begin{proof}
 To prove that $u(\mubar)$ strongly converges to  $u_\hom$, we consider $w_\es:=u(\mubar)-u_\hom $ and proceed using energy estimates similar to the proof of Theorem \ref{Theorem:consistancy}. By \eqref{mu-wave-equation} and \eqref{hom-equation}, $w_\es$ solves for every $t \in (0,T)$ and all test functions $\phi \in H^1(\T)$,
\begin{align} \label{eq:defect_wave}
    \langle\partial_{tt}w_\es,\phi \rangle + \int_{\T}\mu_{\hom} \nabla w_\es \cdot \nabla \phi \, dx = \int_{\T} (\mu_{\hom} - \mubar) \nabla u(\mubar) \cdot \nabla \phi\,dx.
\end{align}
Choose $\phi = \phi(t)$ to be the solution of
\begin{align*}
    -\div( \mu_{\hom} \nabla \phi(t)) = \partial_t w_\es(t), \quad \int_{\T} \phi(t) \, dx = 0.
\end{align*}
Then, we have the identities
\begin{align*}
    \langle \partial_{tt} w_\es, \phi \rangle = \frac12 \frac{d}{dt} (\mu_{\hom} \nabla \phi, \nabla \phi) \quad \text{and} \quad \int_{\T} \mu_\hom \nabla w_\es \cdot \nabla \phi \, dx = \frac12 \frac{d}{dt} \| w_\es \|_{L^2(\T)}^2.
\end{align*}
Now let $p_\es(t) = \big(\mu_{\hom} \nabla \phi(t), \nabla \phi(t) \big) + \| w_\es \|_{L^2(\T)}^2$. Then, it follows from \eqref{eq:defect_wave}, the aforementioned identities, and the uniform bound from Proposition \ref{lem:uniform_bound} that
\begin{align*}
    \frac{d}{dt} p_\es(t) \le | \mu_{\hom} - \mubar | \| \nabla u(\mubar)(t) \|_{L^2(\T)} \| \nabla \phi(t) \|_{L^2(\T)} \le C | \mu_{\hom} - \mubar| \sqrt{p_\es(t)}.
\end{align*}
Here we used that since $\mu_\hom$ is symmetric and positive definite we have for some $\alpha > 0$ that
\begin{align*}
    \| \nabla \phi(t) \|_{L^2(\T)}^2 \le \alpha^{-1} (\mu_\hom \nabla \phi(t), \nabla \phi(t)) \le \alpha^{-1} p_\varepsilon(t).
\end{align*}
As in the proof of Theorem \ref{Theorem:consistancy} we can now conclude
\begin{align*}
    \| w_\es \|_{L^2(0,T;L^2(\T))}^2 \le p_\es(T) \le C |\mu_\hom - \mubar| \rightarrow 0 \quad \text{as } \varepsilon \rightarrow 0.
\end{align*}
This proves $u(\mubar)  \rightarrow u_\hom \quad \text{strongly in } L^2(0,T;L^2(\T))$. For the second assertion, we use triangle inequality and obtain
    \begin{align*}
    \| u_\varepsilon - u(\mubar) \|_{L^2(0,T;L^2(\T))} \leq \| u_\varepsilon -u_\hom\|_{L^2(0,T;L^2(\T))} +\|u_\hom -u(\mubar) \|_{L^2(0,T;L^2(\T))}.
\end{align*}
By Assumption \ref{assumption1}, the first term on the right-hand side vanishes as $\varepsilon \rightarrow 0$. Since we proved $u(\mubar)  \rightarrow u_\hom$ strongly in  $L^2(0,T;L^2(\T))$, the second term also goes to $0$ as $\es\rightarrow 0$. This proves the claim.
\end{proof}

Finally, we conclude weak convergence in $L^2(0,T;L^2(\T))$ of the first space and time derivatives.

\begin{corollary}
    Suppose that Assumption \ref{assumption2}-\ref{assumption1} hold. Then for $\es\rightarrow 0$, we have
\begin{itemize}
    \item[i)] $\nabla (u(\mubar)- u_\varepsilon) \rightharpoonup 0$ weakly in $(L^2(0,T;L^2(\T)))^d$,
    \item[ii)] $\partial_t (u(\mubar) - u_\varepsilon)\rightharpoonup 0$ weakly in $L^2(0,T;L^2(\T))$.
\end{itemize}
\end{corollary}
\begin{proof}
    From Assumption \ref{assumption2} and Proposition \ref{lem:uniform_bound}, we have the uniform bounds
    \begin{align*}
        \|\nabla u(\mubar)-\nabla u_\varepsilon\|_{L^2(0,T);L^2(\T))}\leq C\qquad \|\partial_t u(\mubar)-\partial_t u_\varepsilon\|_{L^2(0,T);L^2(\T))}\leq C,
    \end{align*}
    for some $C>0$ independent of $\es$. Then, there are $w_0\in (L^2(0,T;L^2(\T)))^d$ and $w_1\in L^2(0,T;L^2(\T))$ such that, as $\es\rightarrow 0,$
    \begin{align*}
        \nabla u(\mubar)-\nabla u_\varepsilon &\rightharpoonup w_0\quad \mbox{weakly in } (L^2(0,T;L^2(\T)))^d,\\
        \partial_t u(\mubar)-\partial_t u_\varepsilon &\rightharpoonup w_1\quad \mbox{weakly in } L^2(0,T;L^2(\T)).
    \end{align*}
By Theorem \ref{theorem:strong_convergence_of_u_epsilon}, it follows $w_0=0$ and $w_1=0$, which proves the claim.
\end{proof}

\section{Numerical results}
\label{sec:numerical_results}

In this section, we verify our theoretical findings and further assess the proposed optimization approach in non-periodic settings that are not covered by our analysis. In Sections \ref{sec:ex1} and \ref{sec:ex2d}, we consider space--time modulations of the form \eqref{ass:periodic} in one and two space dimensions. In this setting, a limiting system satisfying Assumption \ref{assumption2} exists, and we demonstrate that the optimization approach recovers the homogenized matrix as $\varepsilon \to 0$.

In Section \ref{sec:ex2}, we study a spatially homogeneous but strongly time-modulated coefficient. Although a rigorous homogenization result is not available in this case, formal homogenized equations were derived in \cite{DV26}. We show that the homogenized coefficients obtained are accurately recovered by our optimization approach.

Finally in Section \ref{sec:ex3}, we consider a non-periodic setting with a piecewise constant temporal modulation, where the modulation intensity is determined by random variables, thereby breaking periodicity. We demonstrate that our approach still provides an effective description of the material properties, highlighting its applicability beyond periodic media and supporting its use in more general, unstructured regimes. \\
The implementation of the experiments is available as a MATLAB code on
\begin{center}
\url{https://github.com/cdoeding/OptimizationTimeVaryingMedia}.
\end{center}

\subsection{Periodic coefficient with weak time modulation $d = 1$}
\label{sec:ex1}

For the first model problem we consider \eqref{wave-equation} with $d = 1$, $T = 0.5$ and the space-time coefficient from \eqref{ass:periodic} that is
\begin{align*}
    \mu_\es(x,t) = \mu_0(x/\es) + \es \mu_1(t/\es), \quad 0 < \varepsilon \ll 1
\end{align*}
with
\begin{align*}
    \mu_0(y) = 2 + \sin(2\pi y), \quad \mu_1(\tau) = \sin(2 \pi \tau).  
\end{align*}
As discussed in Section \ref{sec:asymptotic_analysis}, for this parameter set the system \eqref{wave-equation} satisfies Assumption \ref{assumption1}, cf. \cite{CasadoDiaz2014Homogenization}, with the homogenized coefficient
\begin{align} \label{eq:hom_ex1}
    \mu_\hom = \Big( \int_0^1 \mu_0(y)^{-1} \, dy \Big)^{-1} =  \sqrt{3}.
\end{align}
We emphasize that the modulation of the coefficient $\mu_\es$ in time is only weakly enforced due to the $\es$-scaling of the amplitude in the time modulation. This scaling allows one to derive energy estimates of the solution, which enables one to rigorously find the homogenization limit, e.g., by two-scale convergence. For the initial data we set
\begin{align*}
    v_0(x) = \psi(x,0), \quad v_1(x) = \partial_t \psi(x,0), \quad \psi(x,t) = e^{- \frac{1}{200}(t - x)^2}.
\end{align*}
and note that Assumption \ref{assumption2} is satisfied.\\

We apply our optimization approach to the given setting to recover the homogenized coefficient $\mu_\hom$ from \eqref{eq:hom_ex1}. Note that since $d = 1$ the optimal coefficient is given by \eqref{minimum}. We calculate the optimal constant coefficient for the choices 
\begin{align} \label{eps_values}
    \varepsilon = \kappa^{-1}, \quad \kappa \in \{ 10, 20, 40, 60, 80, 100, 120, 160, 200  \}.
\end{align}
The assumed to be given reference solution $u_\es$ as well as the solution to the standard wave equation introduced in \eqref{minimum} through the inverse operator $\Box^{-1}$ are calculated by numerical approximations. For the space discretization we use linear Lagrange finite elements in space on an equidistant mesh of mesh size $h = \tfrac{1}{1024}$ and for the time integration we use a implicit midpoint approximation with time step size $\Delta t = \tfrac{1}{2048}$. For  comparison, we calculate the homogenized solution $u_\hom$ using the same discretization in space and time. \\

\begin{figure}[t]
    \centering
    \includegraphics[width=0.32\linewidth]{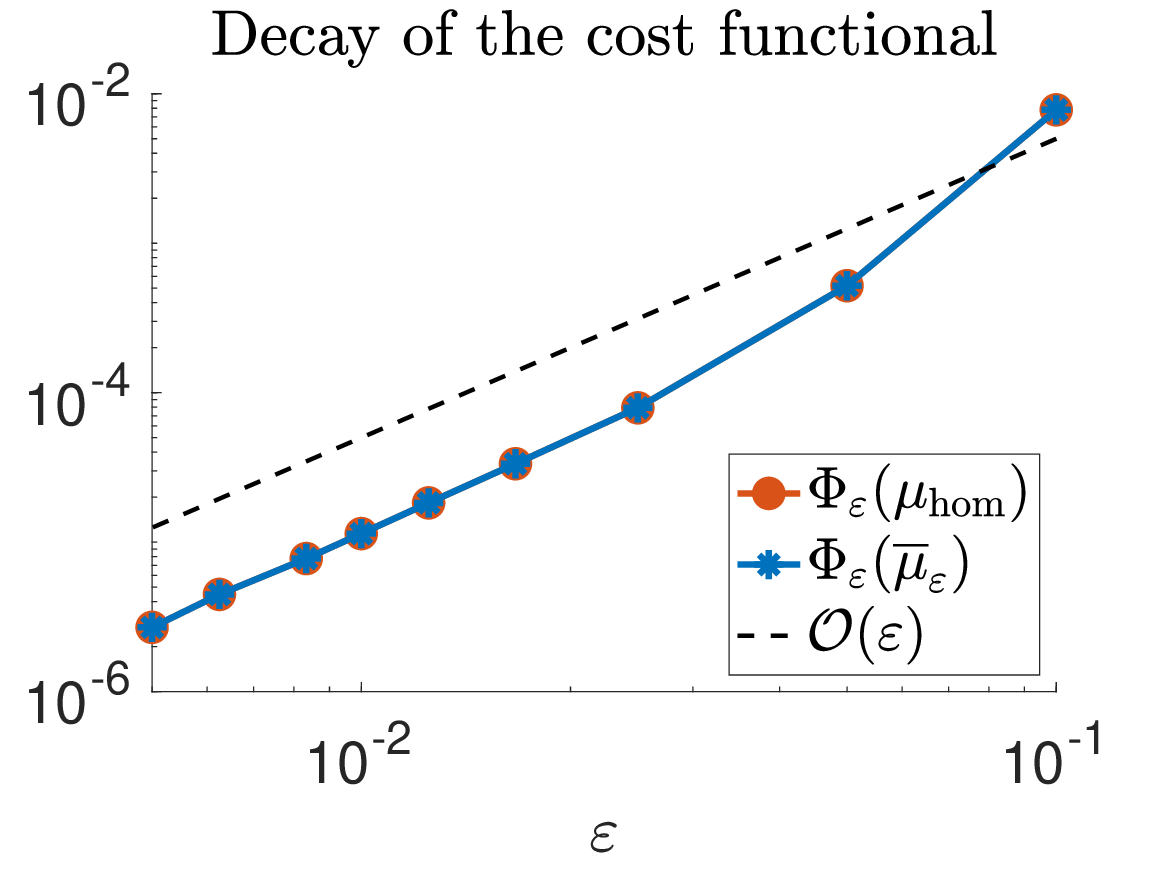}
    \includegraphics[width=0.32\linewidth]{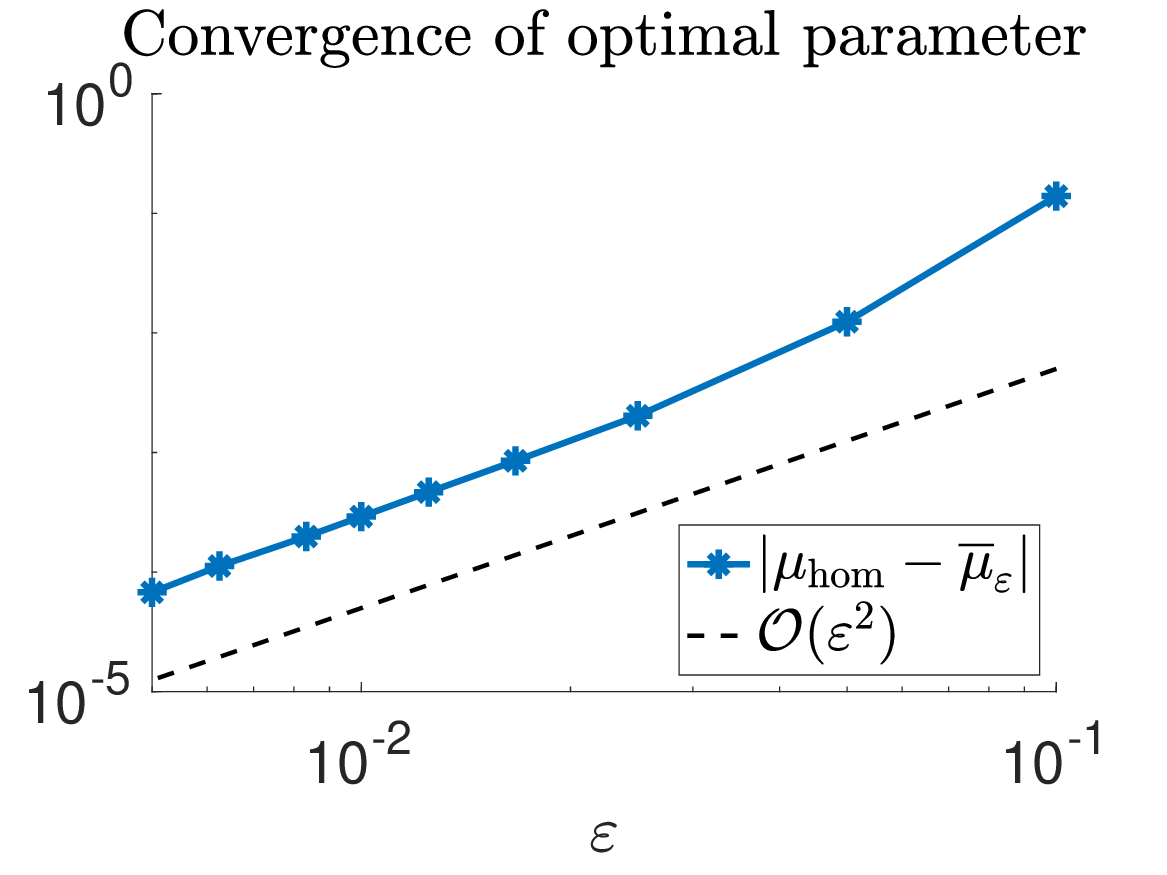}
    \includegraphics[width=0.32\linewidth]{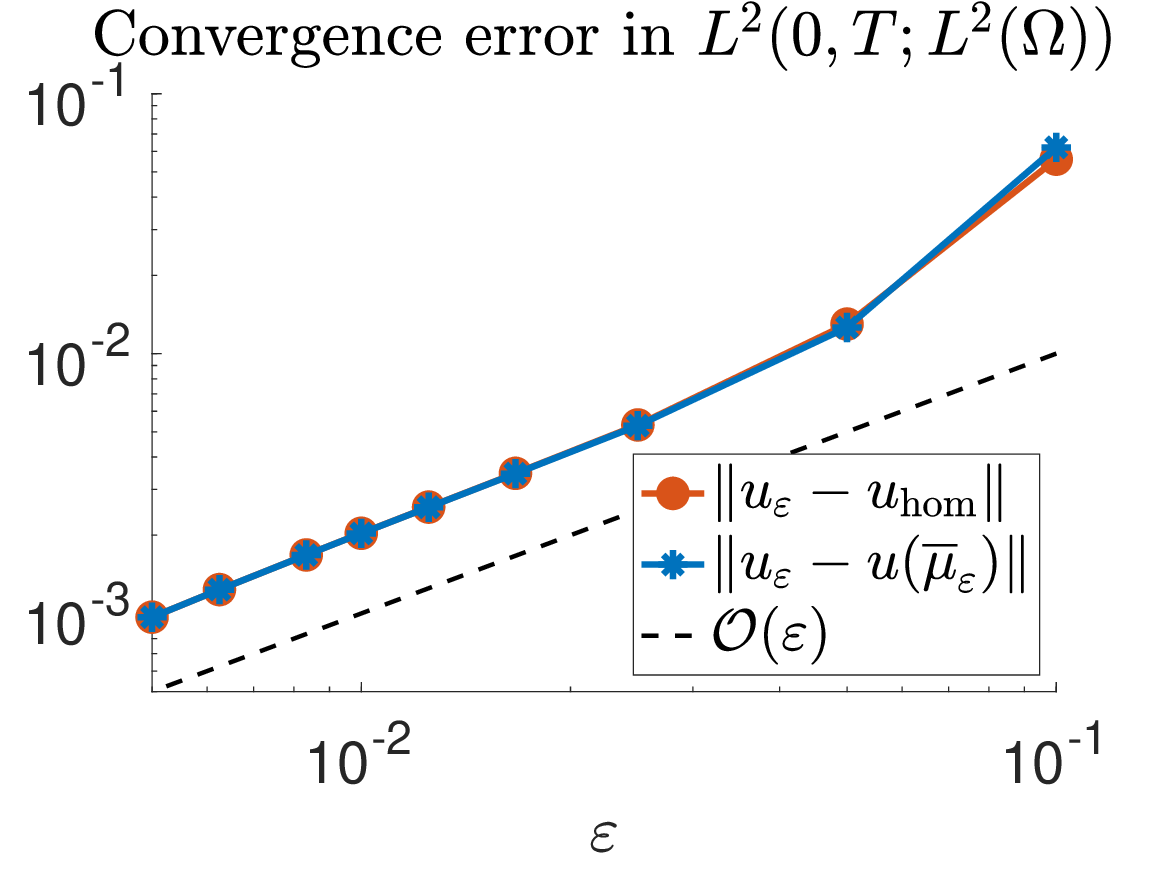}
    \caption{Numerical results for the model problem of Section \ref{sec:ex1} with weak time modulation in $d = 1$. Left: The cost functional evaluated at the optimal parameter $\Phi_\es(\mubar)$ and at the homogenized coefficient $\Phi_\es(\mu_\hom)$. Middle: Convergence of the optimal parameter $\mubar$ towards $\mu_\hom$. Right: Convergence of the errors $\| u_\es - u(\mubar) \|_{L^2(0,T;L^2(\T))}$ and $\| u_\es - u_\hom \|_{L^2(0,T;L^2(\T))}$.} \label{num_ex1}
\end{figure}

The results are presented in Figure \ref{num_ex1}. In the left plot of Figure \ref{num_ex1}, we observe that the minimum of the cost function $\Phi_\es(\mubar)$ decays as predicted by Theorem \ref{Theorem:consistancy}. In particular, we obtain a first-order decay with respect to the microscopic length scale $\es$. The middle plot of Figure \ref{num_ex1} demonstrates that the optimization approach successfully recovers the homogenized coefficient $\mu_\hom$, in agreement with Theorem \ref{theorem:convergence_parameter}. Specifically, we observe second-order convergence of the optimized parameter $\mubar$ towards $\mu_\hom$ as $\es \to 0$. Finally, the right plot of Figure \ref{num_ex1} depicts the corresponding convergence of the solutions w.r.t.~$L^2(0,T;L^2(\T))$. Here we obtain a clear first order decay of the error. Since the optimized coefficient converges to the homogenized coefficient as $\es \to 0$, the solution obtained using the optimized coefficient closely matches the exact solution, with an accuracy comparable to that of the homogenized solution. Overall, this experiment confirms our theoretical results and demonstrates the expected approximation properties of the proposed approach in the periodic setting with weak temporal modulation.

\subsection{Periodic coefficient with weak time modulation $d = 2$}
\label{sec:ex2d}

We now apply the optimization approach to a two-dimensional model problem similar to the previous one-dimensional problem from Section \ref{sec:ex1}. We set $d = 2$, $T = 0.5$, and choose a space-time coefficient with a weak time modulation according to \eqref{ass:periodic} with
\begin{align*}
    \mu_0(y) = \frac{1}{4} \begin{pmatrix}
        2 + \sin(2\pi y_1) & 0 \\ 0 & \big( 2 + \sin(2\pi y_2)\big)^{-1}
    \end{pmatrix}, \quad \mu_1(\tau) = \frac{1}{4} \sin(2 \pi \tau) \, \mathrm{Id}. 
\end{align*}
Again, due to the weak modulation in time, the systems converges to a limiting system satisfying Assumption 
\ref{assumption1} with the homogenized coefficient matrix, cf. \cite{CasadoDiaz2014Homogenization},
\begin{align*}
  \mu_\hom = \frac14 \begin{pmatrix}
      \left(\int_0^1 \tfrac{1}{2 + \sin(2 \pi y)} dy \right)^{-1} & 0 \\ 0 & \left( \int_0^1 2 + \sin(2 \pi y) dy \right)^{-1}
  \end{pmatrix} = \frac14 \begin{pmatrix}
      \sqrt{3} & 0 \\ 0 & \tfrac12
  \end{pmatrix}.
\end{align*}
For the initial data we set
\begin{align*}
    v_0(x) = \psi(x,0), \quad v_1(x) = \partial_t \psi(x,t), \quad \psi(x,t) = e^{-\frac{1}{2\sigma^2} (k \cdot (x- x_0) - c t )^2 }
\end{align*}
with $\sigma = 0.01$, $c = 1$ and $x_0 = k = (\tfrac12,\tfrac12)^\top$. In this case Assumption \ref{assumption2} is fulfilled so that the initial data are rich enough to guarantee the existence of an optimal constant coefficient matrix. \\

\begin{figure}[t]
    \centering
    \includegraphics[width=0.32\linewidth]{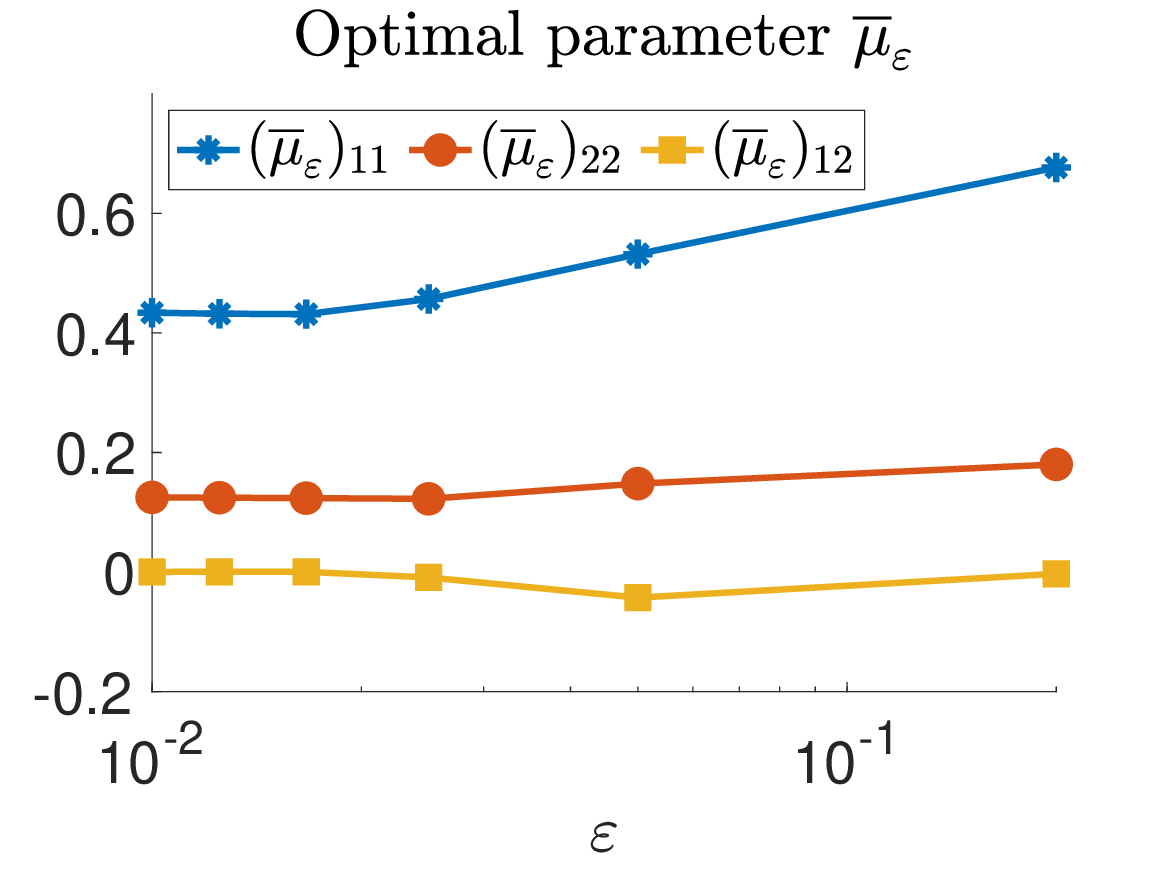}
    \includegraphics[width=0.32\linewidth]{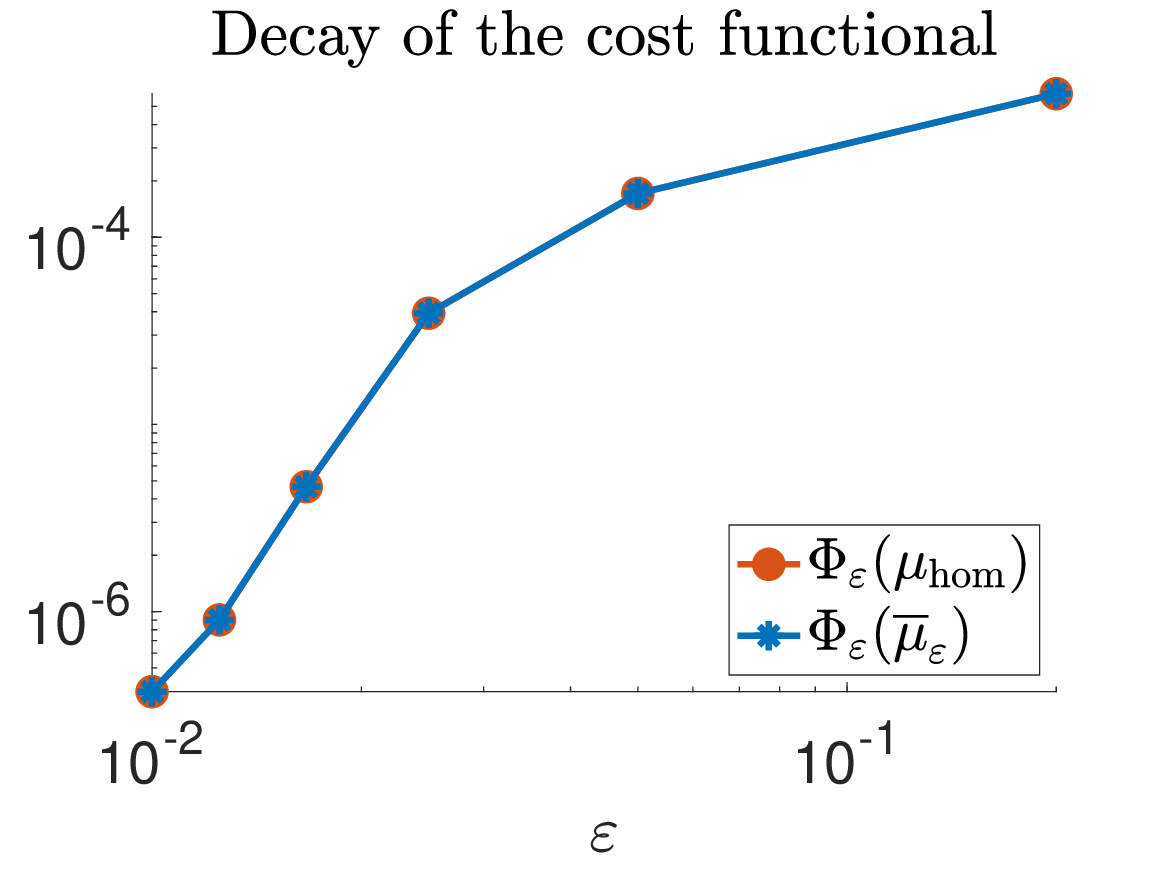} \\
    \includegraphics[width=0.32\linewidth]{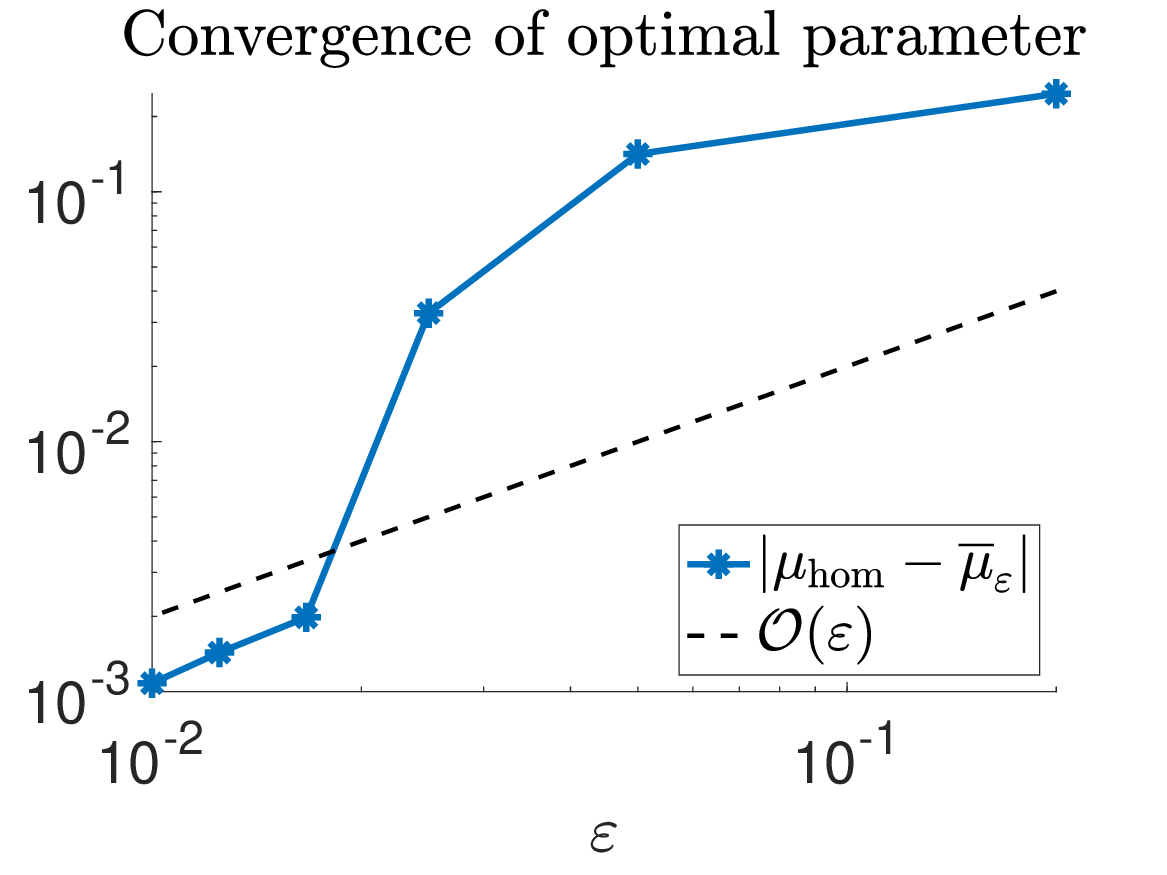}
    \includegraphics[width=0.32\linewidth]{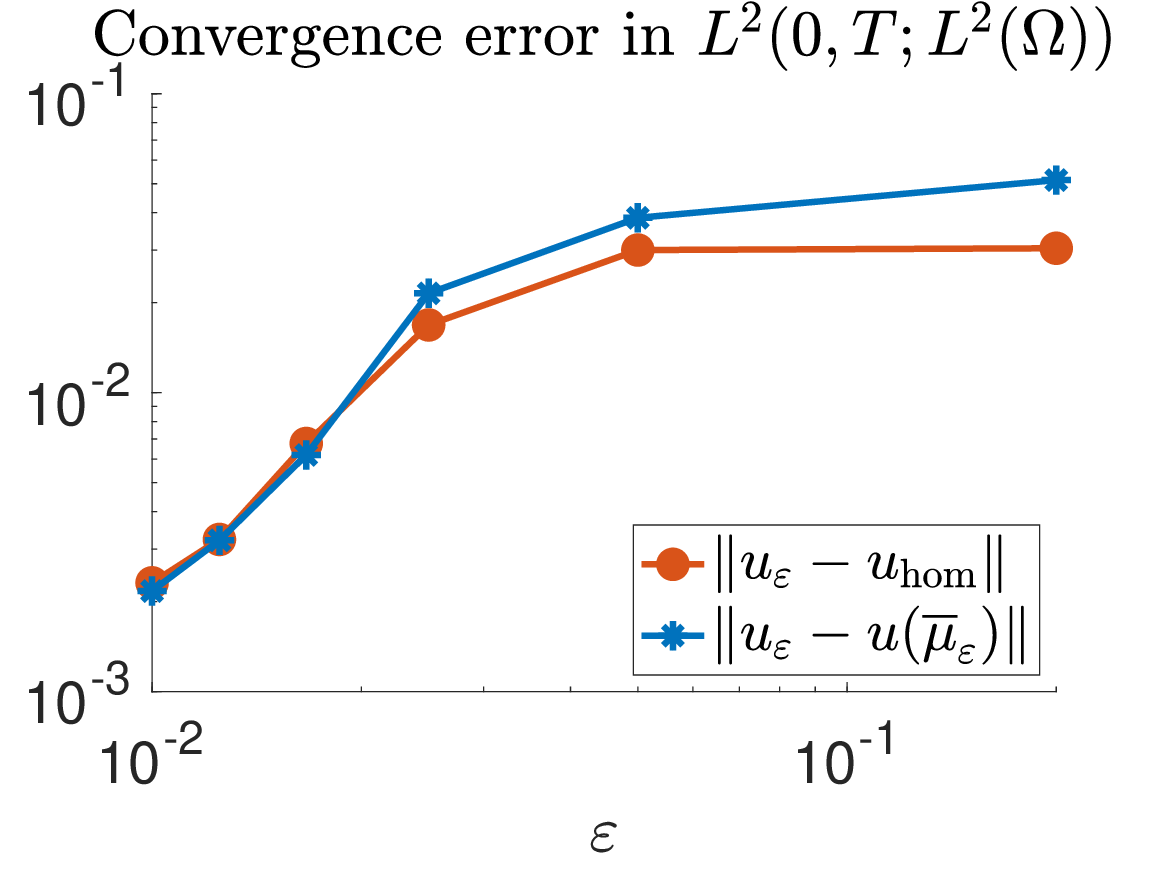}
    \caption{Numerical results for the model problem of Section \ref{sec:ex2d} with weak time modulation in $d = 2$. Top left: The components of the optimal constant coefficient matrix $\mubar$. Top right: The cost functional evaluated at the optimal coefficient matrix $\Phi_\es(\mubar)$ and at the homogenized coefficient matrix $\Phi_\es(\mu_\hom)$. Bottom left: Convergence of the optimal coefficient matrix $\mubar$ towards $\mu_\hom$. Bottom right: Convergence of the errors $\| u_\es - u(\mubar) \|_{L^2(0,T;L^2(\T))}$ and $\| u_\es - u_\hom \|_{L^2(0,T;L^2(\T))}$.} \label{num_ex2d}
\end{figure}

We apply the proposed optimization approach to compute the optimal constant coefficient matrix $\mubar \in \R^{2 \times 2}_{\mathrm{sym}}$ for the length scales $\varepsilon > 0$ specified in \eqref{eps_values}. The optimal coefficient matrix is determined according to \eqref{minimum_if_pos}, following the procedure described in Section~\ref{sec:optimization_approach}. Throughout the computations, we observed that the coefficient matrix $A$ from \eqref{eq:coeff} is positive definite, implying that the minimizer $\mubar$ is unique. \\

For the numerical discretization, we employ linear Lagrange finite elements on a two-dimensional uniform triangulation with mesh size $h = \tfrac{1}{1024}$. Time integration is performed using the implicit midpoint method with time step size $\Delta t = \tfrac{1}{1024}$. \\

The numerical results for the two-dimensional problem are presented in Figure \ref{num_ex2d}. The top-left plot shows the three independent entries of $\mubar \in \mathbb{R}^{2 \times 2}_{\mathrm{sys}}$ and illustrates their convergence to the corresponding entries of the homogenized coefficient matrix $\mu_{\hom}$. This convergence is quantified in the bottom-left plot of Figure \ref{num_ex2d}, where we plot the error between the optimized coefficient and the homogenized coefficient w.r.t.~the $\ell^\infty$-matrix norm. After an initial pre-asymptotic regime, we observe first-order convergence of the error in $\es$. The top-right plot of Figure \ref{num_ex2d} shows the decay of the cost functional. The values are nearly indistinguishable from those obtained by evaluating the cost functional at the homogenized coefficient, indicating that the optimized coefficient achieves a comparable level of accuracy. Finally, the bottom-right plot of Figure \ref{num_ex2d} depicts the error between the solution of the wave equation with the optimized coefficient to the solution $u_\es$ in $L^2(0,T;L^2(\T))$. Although convergence is observed, we cannot clearly identify a corresponding convergence rate. We note that a comparison with the homogenized solution yields a slightly smaller error. However, the small difference is not considered to be significant. Overall, this experiment is consistent with our theoretical results and further demonstrates the applicability of the proposed approach in higher dimensions.

\subsection{Homogeneous coefficient with strong time modulation}
\label{sec:ex2}

In the next experiment we consider the system from Section \ref{sec:ex1} and keep all parameters expect for the time-dependent coefficient $\mu_\es$. We now enforce a strong time modulation in the coefficient by choosing
\begin{align*}
    \mu_\es(x,t) = 0.25 + 0.24\sin(2\pi t/\varepsilon).
\end{align*}
In particular, the coefficient is spatially constant but oscillates in time on the small length scale $\varepsilon > 0$. Due to the strongly enforced time modulation a rigorous homogenization theory is to the best of our knowledge unknown due to the lack of energy estimates in time of the time derivative of the solution $u_\es$. Nevertheless, a formal two-scale expansion for this problem was elaborated in \cite{DV26} where the authors derived the formal homogenized coefficient
\begin{align*}
    \mu_{\hom} = \int_0^\varepsilon \mu_\varepsilon(x,t) \, dt = 0.25.
\end{align*}

\begin{figure}[t]
    \centering
    \includegraphics[width=0.32\linewidth]{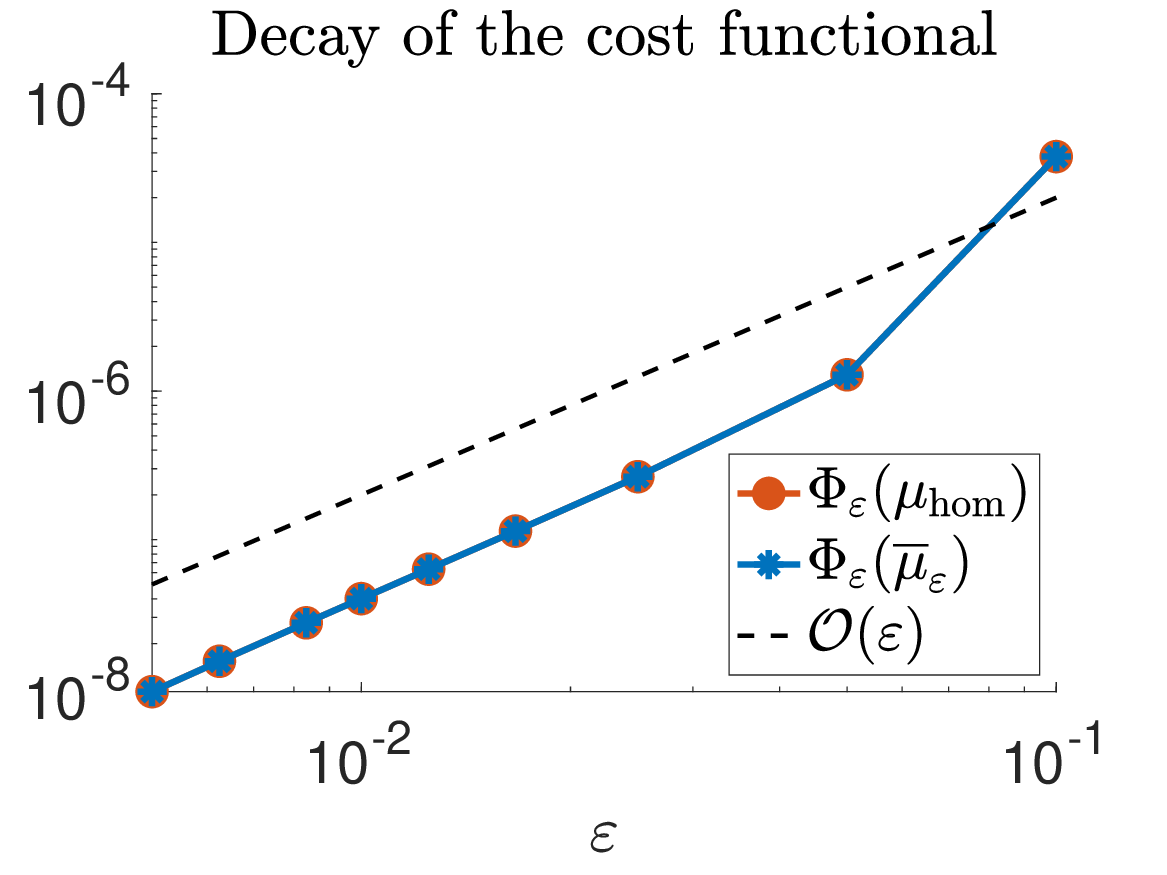}
    \includegraphics[width=0.32\linewidth]{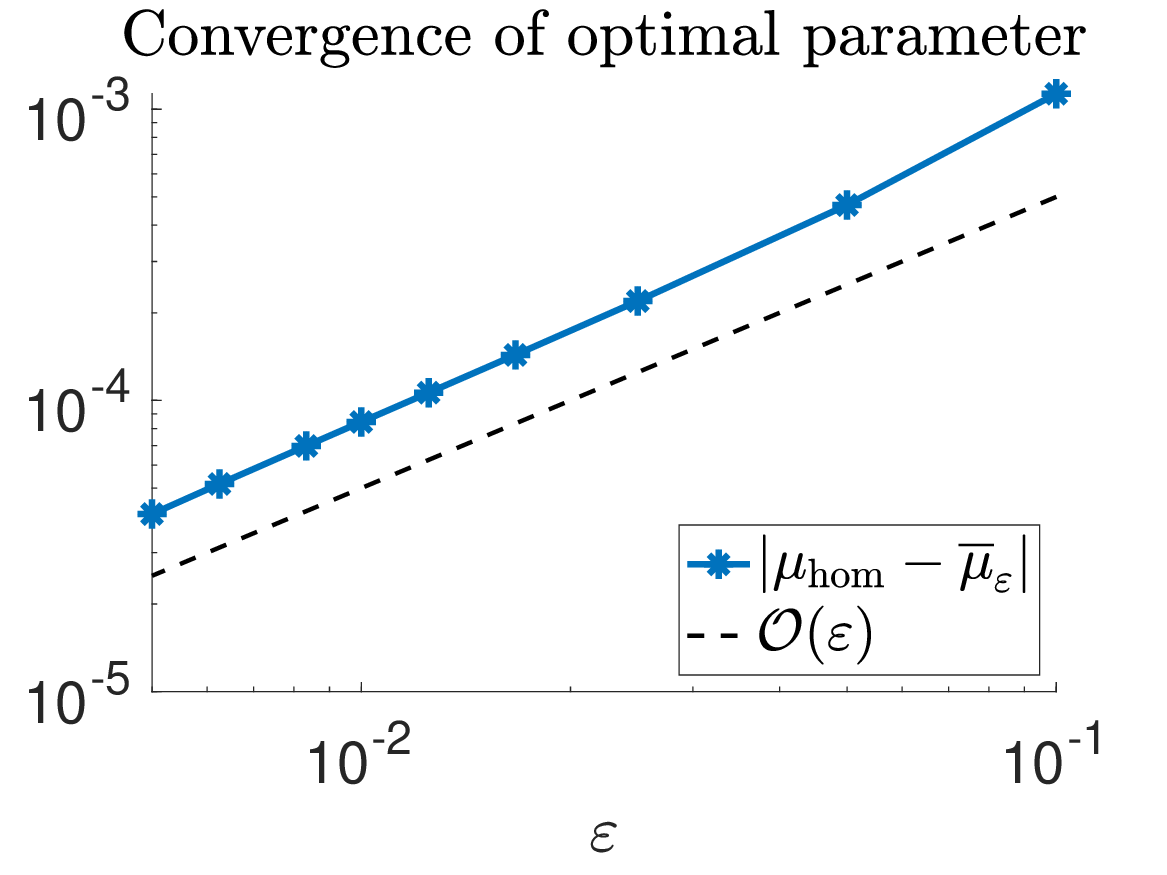}
    \includegraphics[width=0.32\linewidth]{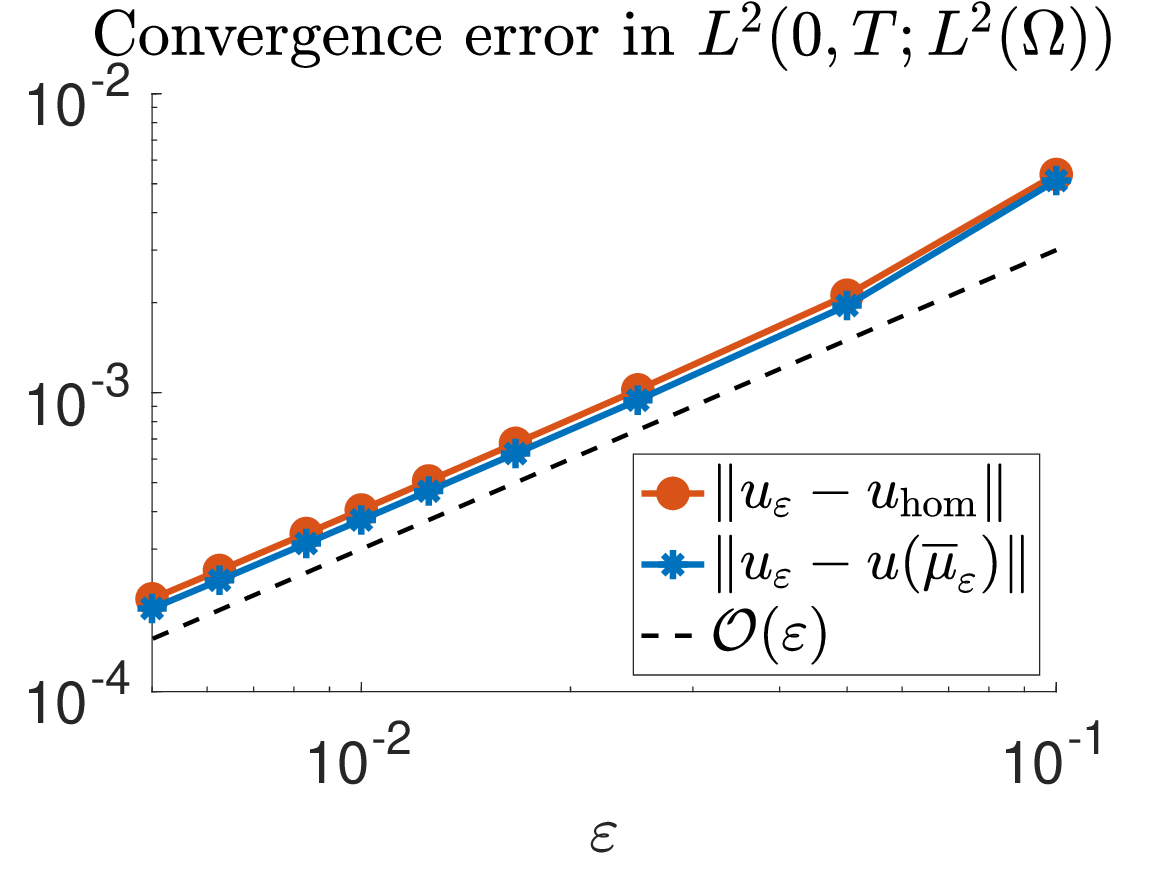}
    \caption{Numerical results for the model problem of Section \ref{sec:ex2} with strong time modulation in $d = 1$. Left: The cost functional evaluated at the optimal parameter $\Phi_\es(\mubar)$ and at the homogenized coefficient $\Phi_\es(\mu_\hom)$. Middle: Convergence of the optimal parameter $\mubar$ towards $\mu_\hom$. Right: Convergence of the errors $\| u_\es - u(\mubar) \|_{L^2(0,T;L^2(\T))}$ and $\| u_\es - u_\hom \|_{L^2(0,T;L^2(\T))}$.} \label{num_ex2}
\end{figure}

The numerical results are presented in Figure~\ref{num_ex2} and exhibit a behavior similar to that observed in Section~\ref{sec:ex1}. As shown in the left plot of Figure~\ref{num_ex3}, the minimum value of the cost functional decays at first order with respect to $\varepsilon$. Moreover, the optimal parameter converges to the homogenized coefficient $\mu_\hom$ at first order, as illustrated in the middle plot of Figure~\ref{num_ex3}. The right plot shows the error between the solution corresponding to the optimally determined constant coefficient and the given solution $u_\es$. Once again, we observe a clear first-order decay. Notably, this error is slightly smaller than the corresponding error between the homogenized solution and the reference solution $u_\es$. \\

Overall, this experiment demonstrates that the proposed approach remains applicable even in problem settings for which a rigorous homogenization theory is unavailable and our assumptions required for the consistency analysis are not satisfied. Nevertheless, the optimization procedure successfully identifies an optimal constant coefficient that accurately captures the dynamics of the reference solution $u_\es$ and which coincides with formally identified homogenized coefficients from \cite{DV26}.

\subsection{Homogeneous coefficient with randomized time modulation}
\label{sec:ex3}

In our final experiment, we break the periodicity of the multiscale coefficient and instead consider a spatially homogeneous, temporally piecewise constant coefficient whose intensity is prescribed by a collection of random variables. More precisely, let $T/\varepsilon = N \in \N$ and define
\begin{align*}
    \mu_\varepsilon(x,t) = \sum_{j=1}^N \beta_j \chi_{I_j}(t),
\end{align*}
where $\beta_j \sim \mathcal{U}(0.1,1)$, $j=1,\ldots,N$, are independent and identically distributed uniform random variables, and $\chi_{I_j}$ denotes the characteristic function of the subintervals $I_j=[(j-1)\varepsilon,\,j\varepsilon)$, each of length $\varepsilon>0$. All other system parameters are again kept as in Section \ref{sec:ex1}. \\

We choose this system as a prototype of an unstructured coefficient for which no homogenization results are available to the best of our knowledge and demonstrate that our optimization approach remains applicable. We apply the optimization procedure using the same numerical discretization as in Section~\ref{sec:ex1}. The results are shown in Figure~\ref{num_ex3} for three independent realizations of the random variables $\beta_j$ for each prescribed length scale $\varepsilon>0$ from \eqref{eps_values}.

\begin{figure}[t]
    \centering
    \includegraphics[width=0.32\linewidth]{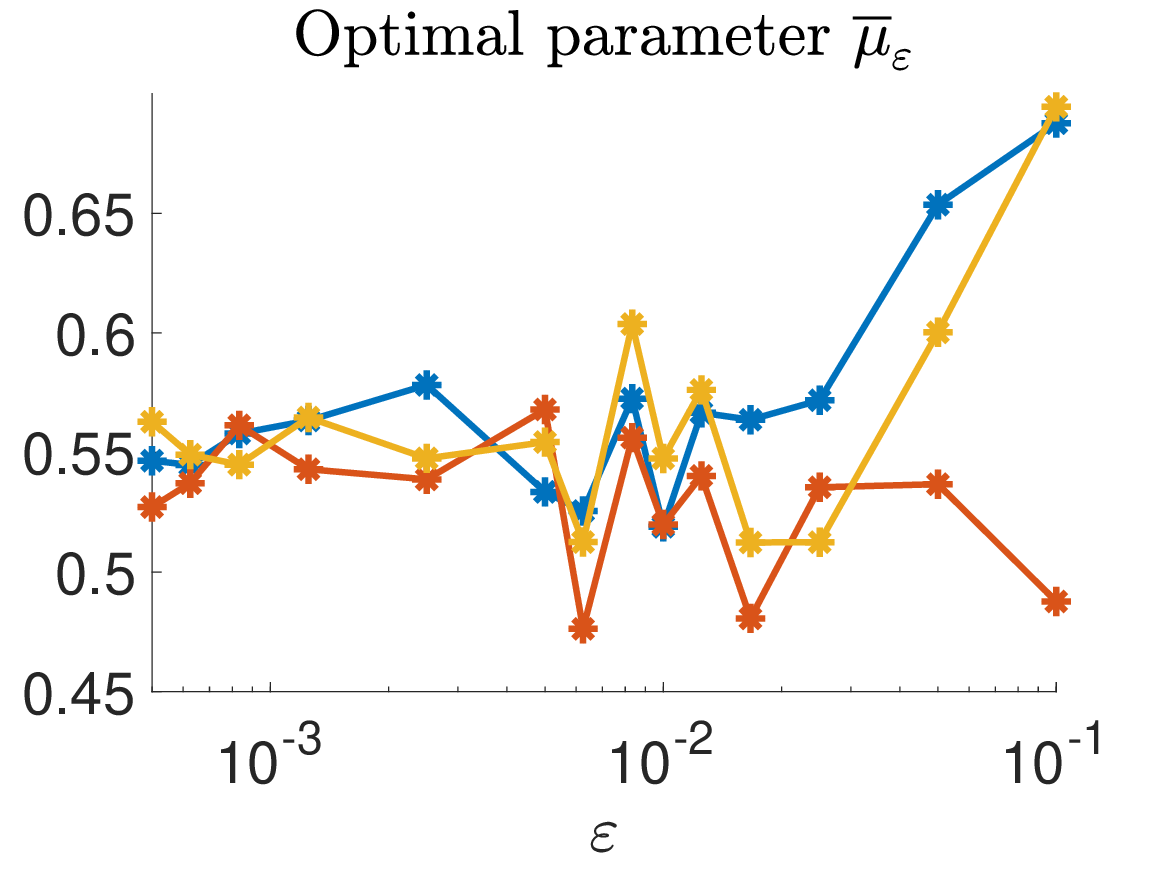}
    \includegraphics[width=0.32\linewidth]{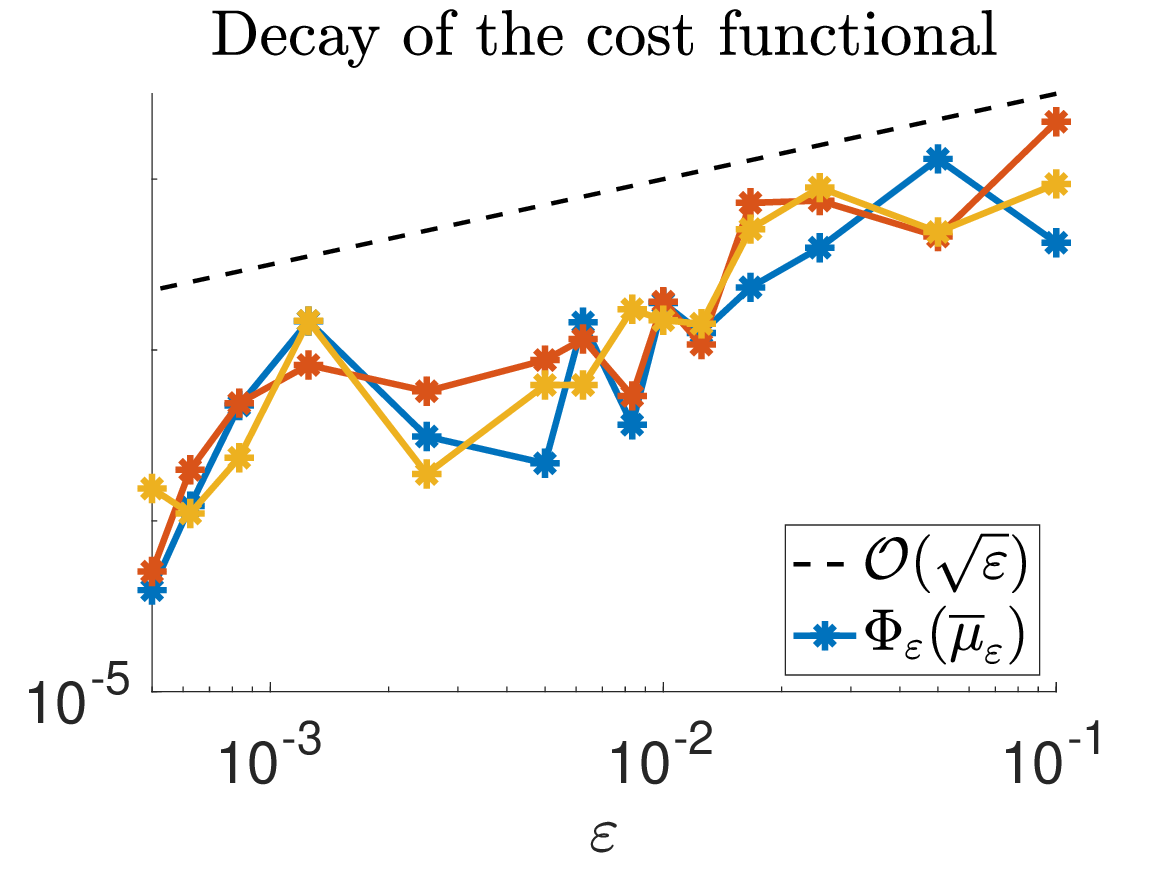}
    \includegraphics[width=0.32\linewidth]{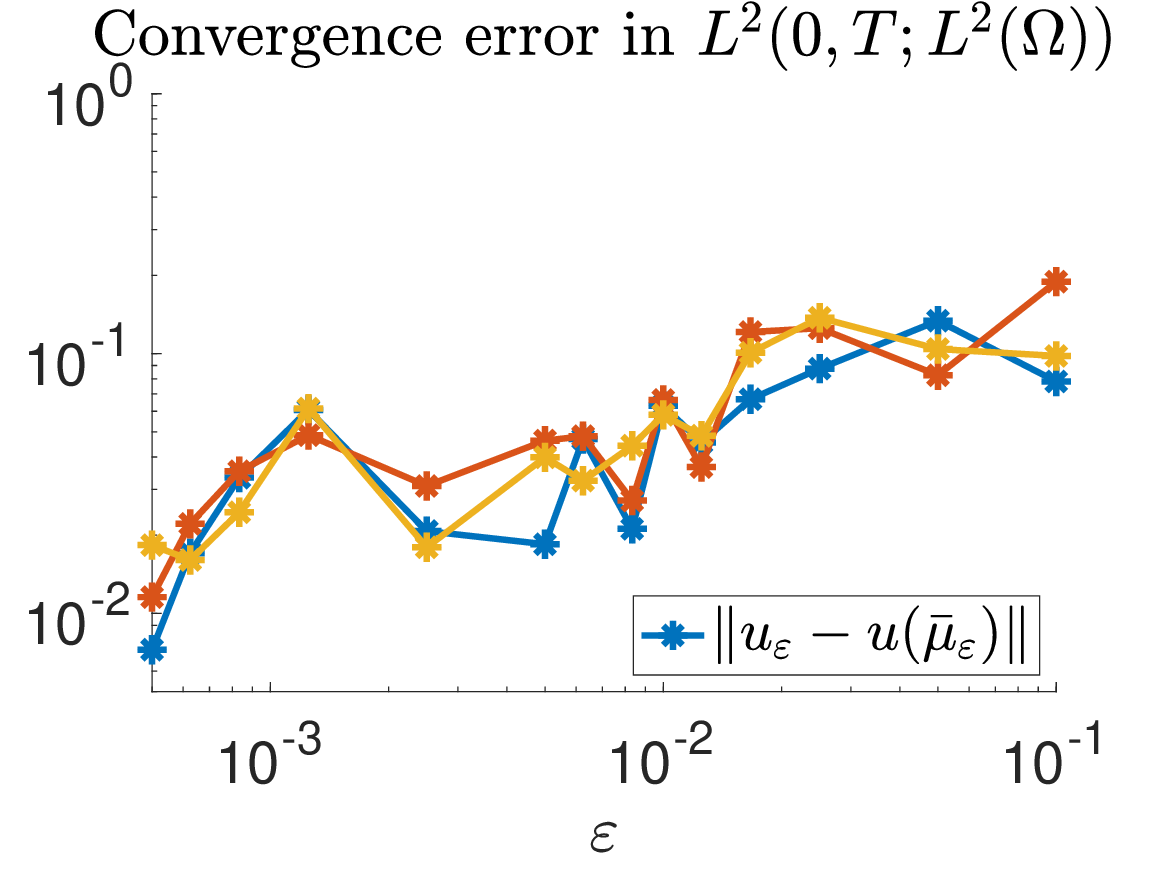}
    \caption{Numerical results for the model problem of Section \ref{sec:ex3} with randomized time modulation in $d = 2$ for three different sample realizations. Left: The optimal coefficient $\mubar$. Middle: The cost functional evaluated at the optimal coefficient $\Phi_\es(\mubar)$. Right: Convergence of the errors $\| u_\es - u(\mubar) \|_{L^2(0,T;L^2(\T))}$.} \label{num_ex3}
\end{figure}

From the left plot of Figure~\ref{num_ex3}, we observe that as $\varepsilon \to 0$ the optimal parameter $\mubar$ approaches a value of approximately $0.55$, which is consistent with the expected value of the random variables $\beta_j$. Moreover, the individual realizations exhibit noticeable variability due to the stochastic nature of the coefficient. The middle plot of Figure~\ref{num_ex3} shows the decay of the cost functional, which appears to converge at a rate close to $\varepsilon^{1/2}$, again in agreement with the expected behavior. Finally, the right panel of Figure~\ref{num_ex3} displays the convergence error of the corresponding solutions, which likewise exhibits a relatively slow decay. Overall, this experiment further demonstrates that our optimization approach is not restricted to settings in which a homogenized limit is known to exist. Instead, it remains applicable to systems with significantly less structure, such as coefficients without periodicity or any known homogenization theory.

\section{Conclusion}

In this paper, we presented an optimization-based approach for identifying effective coefficients of wave equations with space-time heterogeneous coefficients. Unlike classical homogenization, our method does not require knowledge of the coefficients or their underlying structures beforehand. Instead, it directly determines an effective constant matrix based on given data of the solution to the wave equation. This makes our method attractive for applications where the microscopic material structure or its temporal modulation is unknown.

The effective matrix is obtained by minimizing a suitable cost functional. We have proven the well-posedness of the optimization problem and the asymptotic consistency of the approach. When a homogenized limit exists, the identified matrix converges to the homogenized coefficient as the microscopic scale approaches zero. Additionally, the solution to the wave equation with the optimized constant coefficient converges to the homogenized solution. Our numerical results confirm these observations and demonstrate the applicability of the approach in cases where our analysis is not applicable.

Although we presented our analysis for periodic boundary conditions, the optimization framework and analytical techniques naturally extend to other boundary conditions. Additionally, we assumed that the given solution data were available exactly. In practice, however, measurement data are affected by noise and experimental errors. Future research may focus on extending the present framework to account for such perturbations. Finally, we emphasize that the proposed method is not intended to be a computationally efficient numerical homogenization technique. Rather, it complements existing multiscale methods by addressing the inverse problem of identifying effective models directly from observed wave propagation.

\appendix
\section{Appendix}

\begin{lemma} \label{lem:appendix_uniqueness}
    Let $k_j \in \Z^d$, $j = 1,\dots, d_* := \tfrac{d(d+1)}{2}$ be a collection of vectors such that
    \begin{align*}
        \mathrm{span}\big\{ k_j k_j^\top : j = 1, \dots, d_*\big\} = \Rsym
    \end{align*}
    and let $\mu \in \Rsym$ with
    \begin{align} \label{appendix_coeff_system}
        k_j^\top \mu k_j = 0, \quad j= 1,\dots,d_*.
    \end{align}
    Then, $\mu = 0$ follows.
\end{lemma}

\begin{proof}
    Recall the map $(i,j) \mapsto [i,j]$ from \eqref{enumeration} that enumerates the set of index-tuples $\{ (i,j): 1 \le i \le j \le d \}$. As in \eqref{identification}, we identify the matrix $\mu \in \Rsym$ uniquely with the vector $\bfmu \in \R^{d_*}$. Similarly, we define the vectors $\mathbf{k}^\ell \in \R^{d_\star}$, $\ell = 1,\dots,d_\star$ by
    \begin{align*}
        \mathbf{k}^\ell_{[i,j]} = \begin{cases} (k_\ell k_\ell^\top)_{ii}, & 1 \le i = j \le d \\
        2(k_\ell k_\ell^\top)_{ij}, & 1 \le i < j \le d \end{cases}.
    \end{align*}
    Since the matrices $k_\ell k_\ell^\top$, $\ell = 1,\dots,d_\star$ span the entire space $\Rsym$, it is a simple matter of fact that the vectors $\mathbf{k}^\ell$, $\ell = 1,\dots,d_\star$ are linearly independent. Therefore, the matrix
    \begin{align*}
        K = (\mathbf{k}^1, \dots, \mathbf{k}^{d_\star})^\top \in \R^{d_\star \times d_\star}
    \end{align*}
    is invertible. We compute, for $\ell = 1,\dots,d_*$,
    \begin{align*}
        (K \boldsymbol{\mu})_\ell = \sum_{[i,j] = 1}^{d_\star} K_{\ell,[i,j]} \boldsymbol{\mu}_{[i,j]} = \sum_{i,j = 1}^{d} (k_\ell)_i \mu_{ij} (k_\ell)_j = k_\ell^\top \mu k_\ell.
    \end{align*}
    Thus \eqref{appendix_coeff_system} reads, $K \boldsymbol{\mu} = 0$. Since $K$ is invertible, this proves the claim. 
\end{proof}

\section*{Acknowledgement}
This work is funded by the Deutsche Forschungsgemeinschaft (DFG, German Research Foundation) – Project-ID 258734477 – SFB 1173 and under Germany's Excellence Strategy – EXC-2047/1 – 390685813.

The authors would like to thank Puneet Garg, Michael Plum, and Carsten Rockstuhl for valuable discussions on the topic of this work.

\end{document}